\documentclass{scrartcl}
\usepackage[utf8]{inputenc}
\usepackage[T1]{fontenc}
\usepackage{lmodern}
\usepackage{amsthm}
\usepackage{amsfonts}
\usepackage{mathtools}
\usepackage{stmaryrd}
\usepackage{MnSymbol}
\usepackage{xspace}
\usepackage{mathrsfs}
\usepackage{tikz-cd}
\usepackage{appendix}
\usepackage[french=guillemets]{csquotes}
\MakeOuterQuote{"}
\usepackage[shortlabels]{enumitem}
\setlist{nosep}
\usepackage{xcolor}
\usepackage{comment}

\usepackage[colorlinks=true, breaklinks=true, urlcolor= black, linkcolor= black, citecolor= black, 
bookmarksopen=true,linktocpage=true,plainpages=false,pdfpagelabels]{hyperref}
\usepackage[capitalise]{cleveref}

\theoremstyle{plain}

\newtheorem{theorem}{Theorem}[section]
\newtheorem*{theorem*}{Theorem}
\newtheorem{lemma}[theorem]{Lemma}
\newtheorem{proposition}[theorem]{Proposition}
\newtheorem{corollary}[theorem]{Corollary}
\newtheorem{fact}[theorem]{Fact}
\newtheorem{claim}[theorem]{Claim}

\theoremstyle{definition}

\newtheorem{definition}[theorem]{Definition}
\newtheorem{notation}[theorem]{Notation}

\theoremstyle{remark}

\newtheorem{remark}[theorem]{Remark}

\newcounter{claimqed}
\newcommand\claimqed{\setcounter{claimqed}{0}%
\leavevmode\unskip\penalty9999
\hbox{}\nobreak\hfill\quad\hbox{$\square_{\text{Claim}}$}}
\newenvironment{claimproof}[1]{\par\noindent\textit{Proof.}\space#1\setcounter{claimqed}{1}}{\ifnum\value{claimqed}=1\claimqed\fi}

\newcommand\A{A}
\newcommand\B{B}

\newcommand\I{\mathrm{I}}

\newcommand\p{p}

\newcommand\Ri{\mathbb{R}_{\infty}}

\newcommand\acl{\mathrm{acl}}
\newcommand\dcl{\mathrm{dcl}}

\newcommand{\tp}{\operatorname{tp}}
\renewcommand\epsilon{\varepsilon}

\renewcommand\mid{:}

\newcommand{\rest}{\upharpoonright}

\newcommand{\St}{\mathrm{St}}

\newcommand\ACVF{\mathrm{ACVF}}

\newcommand{\U}{\mathbb M}
\newcommand{\kres}{\mathbf k}

\newcommand{\defst}{\mathrm{DEF\text{-}ST}}

\newcommand{\an}{\mathrm{an}}
\newcommand{\M}{\mathbb M}

\newcommand{\DEFST}{\mathrm{DEF\text{-}ST}}

\newcommand{\R}{\mathbb R}

\newcommand{\N}{\mathbb N}

\title{Radon measures in the Berkovich space and generically stable measures in $\mathrm{ACVF}$}
\author{Alf Onshuus\thanks{}
\and Mariana Vicar\'ia\thanks{The author was supported by the Humboldt Foundation (Postdoctoral Fellowship).}}

\begin{document}

\maketitle
\section{Introduction}
In \cite{tamegeo}, Hrushovski and Loeser constructed a model theoretic analogue of the Berkovich analyfication: given an algebraic variety $X$ defined over a spherically complete algebraically closed valued field $M$, whose value groups is $(\mathbb{R},+, \leq,0)$, they constructed its stable completion $\widehat{X}(M)$, which is a strict-pro-definable set, and consists of the generically stable types centered in $X$; and they endowed it with an ind-definable topology and proved it was homeomorphic Berkovich analyfication $X^{an}$ defined by Berkovich in \cite{Berkovich}. One of the results in this paper is to extend this correspondence from types to measures. We prove that, when $M$ is spherically complete and its
value group is $\mathbb R$, Radon probability measures on $X^{\an}$
correspond canonically to global $M$-invariant generically stable
Keisler measures concentrating on $X$. Under this correspondence, the
Dirac measure at a point $\xi\in X^{\an}$ corresponds to the
generically stable type represented by $\xi$. Thus the
Hrushovski--Loeser correspondence between individual analytic points
and tame types appears as the atomic case of a more general
measure-theoretic correspondence.\\
To achieve this we first have to study and prove alternative characterizations of generically stable measures in $\ACVF$, results that generalize the theorem from \cite{HHM} for global $M$ invariant types: a global $M$-invariant type $p$ is generically stable, if and only it is stably dominated if and only it is orthogonal to the value group.  For this purpose, we needed to introduce measure-theoretic analogues of stable domination and orthogonality to
the value group, which are instead expressed through uniqueness of extensions. These results were heavily inspired in Keisler's paper \cite{Keisler}.\\
In the context of valued fields, let $M\models\ACVF$, and let $\mu$ be a global $M$-invariant Keisler measure. We write $\mathbb{M}$ to denote a monster model of $\mathrm{ACVF}$, and $\Gamma$ to indicate the value group. We say that $\mu$ is \emph{orthogonal to the value group} if, for every finite tuple $e$ from $\Gamma(\mathbb{M})$, the measure
$\mu\rest_M$ has a unique extension to $Me$. We say that $\mu$ is
\emph{stably dominated over $M$} if the restriction of $\mu$ to $M$-definable stable and stably embedded sets has a
unique extension to the full Boolean algebra of definable subsets, and
this extension is $\mu$. We can now state the first theorem of this paper, which answers a question of Gannon and Ye (see Question $1.1$ in \cite{extension}).
\begin{theorem}
\label{thm:introduction-main}
Let $M$ be a spherically complete model of $\ACVF$, and let $\mu$ be a
global $M$-invariant Keisler measure. The following conditions are
equivalent:
\begin{enumerate}
\item $\mu$ is stably
dominated over $M$;
\item $\mu$ is generically stable over $M$;
\item $\mu$ is orthogonal to the value group over $M$.
\end{enumerate}
\end{theorem}
 For the direction $(1)$ implies $(2)$ we show that if  the measure is stably dominated over $M$ then it is definable and finitely satisfiable, the measure inherits such properties from its pushforwards to the stable part of the structure. For the direction $(2)$ implies $(3)$, we use distal strong honest definitions for the value group, together with definability
and finite satisfiability of the measure, to obtain arbitrarily close
$M$-definable inner and outer approximations to formulas with
value-group parameters. The proof is close to \cite[Theorem $9.21$]{simon}.
For the direction $(3) \rightarrow (1)$ we use the fact that, over a spherically complete
base, invariant types are determined by their restrictions to the
base together with the types in residue-field and value-group given by taking the pro-definable pushforward of their definable closure (see \cite[Theorem 12.15]{HHM}). Orthogonality to the value group allows us to approximate the value-group
conditions by formulas over $M$. A compactness argument on the
support of the measure then yields inner and outer approximations
to arbitrary definable sets by elements of the stable Boolean
algebra. These approximations imply uniqueness of extension from
the stable part. We therefore obtain stable domination.\\

Using \cref{thm:introduction-main} we prove a 
measure-theoretic extension of the Hrushovski--Loeser correspondence:

\begin{theorem}
\label{thm:introduction-correspondence}
Let $M\models\ACVF$ be spherically complete, suppose that $\Gamma(M)=(\mathbb R,+,\leq,0),$
and let $X$ be a projective variety defined over $M$. There is a canonical bijection between Radon probability measures on $X^{\an}$ and global
$M$-invariant generically stable Keisler measures concentrating on
$X$.
\end{theorem}
The proof builds up on the correspodence of Borel measures in $(\R_{\infty},+, \leq, 0)$ and smooth measures over $(\R_{\infty},+,\leq,0)$ (\cref{theo: radon and smooth in R}) where $\R_{\infty}$ denotes $\R$ extended by a constant element $\infty$, declared to be above any other element in the value group. Such characterization builds on previous work for the real closed field see \cite[Proposition $5.5$]{SimonGenericallyStable} and \cite[Theorem $6.3$]{genmeasures}. More precisely, we show:
\begin{theorem}\label{theo: intro radon and smooth in R}
Let $M$ be an o-minimal expansion of $(\R_\infty, +, \leq, 0)$. 
\begin{enumerate}
\item If $V$ is an $\R_{\infty}$-definable set contained in a closed box, and $\mu$ a countably additive Keisler measure on $V$ over $\R_{\infty}$ (in the sense of \cite[Definition $6.1$]{genmeasures}), then $\mu$ is smooth. In particular, the Keisler measure on the $\R_{\infty}$-definable subsets of a box $B$ induced by a Borel measure on $B$ is smooth.
\item Conversely, if $\mu$ is a smooth measure over $(\R_\infty)^k$ concentrating in a closed box $B$, there is a $\sigma$-additive $\R_{\infty}$-Borel measure on $B(\R_\infty)$ such that $\mu_0$ induces $\mu$ as a Keisler measure.
\end{enumerate}
\end{theorem}
Now we can sketch briefly the key components on the proof of \cref{thm:introduction-correspondence}. Given a generically stable measure $\mu$ concentrating on $X$ we use Payne's tropicalization theorem (\emph{c.f} \cite[Theorem $4.2$]{Payne09}) to construct a Radon measure on $X^{an}$ from the pushforwards of $\mu$ in the projective tropicalized spaces. Each such pushforward is smooth since the value group is stably embedded and distal, and \cref{theo: intro radon and smooth in R}.(2) allows us to extend it to a Borel measure on the tropicalized space, which is compact and Hausdorff. We therefore obtain a compatible system of Borel measures on the tropicalized spaces, and we then lift it to a measure on the inverse limit of the tropicalized spaces by Prokhorov's theorem \cite[Theorem $5.5.1$]{burba9}. The converse direction follows as an application of \cref{thm:introduction-main}, we start with a Radon measure on $X^{an}$, form which we construct a Keisler measure centered in $X$ and we show that it is generically stable arguing that it is orthogonal to the value group. To argue that such measure is orthogonal to the value group, we use the stable embeddedness of $\Gamma$ and use that each of the pushforwards of Radon measure to the value group must be smooth by \cref{theo: intro radon and smooth in R}.(1).\\
An important source of examples is provided by the
non-Archimedean Monge--Amp\'ere measures introduced by
Chambert--Loir; see
\cite{ChambertLoir06,ChambertLoir11}. Let $X$ be a projective variety
of dimension $d$, let $\mathscr X$ be a normal projective model of
$X$ over the valuation ring of $M$, and write its scheme-theoretic
special fiber as $\mathscr X_{\mathbf k}=\sum_i m_iY_i.$ Each irreducible component $Y_i$ determines a divisorial point
$\xi_i\in X^{\an}$. Suppose that model metrics on line bundles
$L_1,\ldots,L_d$ are induced by line bundles
$\mathscr L_1,\ldots,\mathscr L_d$ on $\mathscr X$, with respective
exponents $e_1,\ldots,e_d$. The associated Chambert--Loir measure is
the finite atomic measure

$c_1(\overline L_1)\cdots c_1(\overline L_d)\delta_X
=
\frac{1}{e_1\cdots e_d}
\sum_i
m_i
\bigl(
c_1(\mathscr L_1)\cdots c_1(\mathscr L_d)
\mid Y_i
\bigr)
\delta_{\xi_i}.$

Thus the special fiber determines the divisorial points supporting the
measure, while intersection theory determines their weights. After
normalizing by its total mass, this measure is a Radon probability measure. In particular, model measures correspond to finite
weighted averages of the generically stable types represented by the
divisorial points $\xi_i$, and therefore they are generically stable. 
It is more challenging to see that the Monge-Amp\'ere measures that arise as weak limits of the previous model measures are also generically stable. 
Indeed, Chambert--Loir measures associated with continuous
semipositive metrics are obtained as weak limits of finite atomic
model measures supported on divisorial points. It is in general false that weak limits of generically stable measures are generically stable, however weak limits of Radon measures are Radon measures.  Our correspondence
therefore assigns a generically stable Keisler measure to every such
normalized non-Archimedean Monge-Amp\'ere measure, which are obtained as weak limits of the model Chamber Loir measures.\\

This paper is organized as follows:
\begin{enumerate}
\item Section $2$: Preliminaries. 
\item Section $3$: We
introduce the notions of stable approximation, stable domination, and orthogonality
to a stably embedded sort for Keisler measures (\emph{c.f.} \cref{def:stable-approximation-and-domination},\cref{def:measure-orthogonal-gamma}). The main result of this
section is the characterization theorem stated in
\cref{thm:introduction-main}.
\item  Section $4$: We study the correspodance  borel measures in and smooth measures on $\R_{\infty}$ and prove \cref{theo: intro radon and smooth in R}. 
\item Section $5$:  We establish the correspondence in
\cref{thm:introduction-correspondence}. 
\end{enumerate}

\subsection*{Acknowledgments}
 The second author would like to thank Martin Hils for insightful conversations on the subject and sharing some relevant literature. She would also like to thank Kyle Gannon and Artem Chernikov for mentioning to her the first question that we answered in \cref{thm:introduction-main}.

 \subsection*{AI disclosure}
Most of this paper and its ideas are due to the two humans that wrote it. Chat GPT Sol $6$ was used to edit and format the paper as well as proof read it. We indicate in a detail way its contributions: \\

After we (humans) proved that a Radon measure on the stable completion $\widehat{X}(M)$ induces a generically stable measure centered on $X$ (\emph{c.f.} \cref{thm:radon-barycentre-criterion}),
we had an outline for the proof of a possible converse (\cref{prop:projectivecase}) using Payne's Theorem together with the fact, which we believed known, that any Radon measure induced generically stable measures over compact subsets of the reals. When asking Chat GPT Sol $5.6$ about whether the statement of constructing a Borel measure from a smooth measure on a compact set in the real closed field setting was a known fact, it pointed out the reference of \cite[Proposition $5.5$]{SimonGenericallyStable}.

While proof reading Chat GPT Sol $6$ corrected some of our arguments which we record below:
\begin{itemize}
\item Claim $4.6$ in Theorem $4.3$: the paragraph right after the definition of the distance function is an argument that build in our proof, but Chat GPT Sol $6$ corrected it explaining that we simply needed $st^{-1}(X_{J}) \cap C_{J} \setminus X_{J} \subseteq d_{X_{J}}^{-1}(0^{+})$.
\item Pointed out the need of Fact $2.18$ and its reference to argue that Borel measures are unique in Proposition $5.20$.
\end{itemize}

Lastly, Chat GPT Sol $6$ was used to fix the enumeration of the present paper as well making the references in the bibliography uniform.

\section{Preliminaries}
\subsection{Model-theoretic notions}

Throughout this subsection, \(T\) is a complete first-order \(\mathcal L\)-theory and \(\U\models T\) is its monster model. We write \(A\) for a small set of parameters and \(M\) for a small elementary submodel of \(\U\).

\begin{definition}\label{def:distal-type}
Let \(\A\subseteq \U\) be small.
\begin{enumerate}
    \item Let \(p(x),q(y)\in S(\A)\). We say that \(p\) and \(q\) are \emph{weakly orthogonal} if
    $ p(x)\cup q(y)$
    determines a complete type over \(\A\).

    \item Let \(p(x)\) and \(q(y)\) be global \(M\)-invariant types. We say that \(p\) and \(q\) are \emph{orthogonal} if they are weakly orthogonal as global types. 
   \item If \(p\) is a global \(\A\)-invariant type. We say that \(p\) is \emph{distal over \(\A\)} if, for every tuple \(d\) and every Morley sequence
$I\models p^{\otimes \omega}\rest_{\A d}$
the types $
p\rest_{\A I}$ and $
\tp(d/\A I)$
are weakly orthogonal.
\end{enumerate}
\end{definition}

\begin{definition}\label{def:stable-set-stably-embedded}
Let \(B\subseteq \U\) be definable.  
\begin{enumerate}
 \item A formula \(\phi(\bar x;\bar y)\) is \emph{stable}, if there are no sequences
    $(a_i,b_i)_{i\in \mathbb N}$
    with $a_i,b_i$ tuples in $\U$ such that
    $\U\models \phi(a_i;b_j)$ if and only if $i<j.$ The formula $\phi$ might have parameters.
  \item We say that \(B\) is \emph{stable in \(\U\)} if, for every formula \(\phi(x;y)\), there are no sequences
    $(a_i,b_i)_{i\in \mathbb N}$
    with \(a_i,b_i\in B\) such that
    $\U\models \phi(a_i;b_j)$ if and only if $i<j.$ 
    \item We say that \(B\) is \emph{stably embedded in \(\U\)} if for every formula \(\phi(x_1,\dots,x_n;y)\), there is a formula \(\psi(x_1,\dots,x_n;z)\), depending only on \(\phi\), such that for every \(b\in \U^y\), there is \(d\in B^z\) satisfying
    $ \phi(B^n,b)=\psi(B^n,d).$
\end{enumerate}
\end{definition}

Notice that a formula $\theta(x)$ which is stable as a partial type, defines a stable set, but that the converse does not hold. 

\begin{fact}\cite[Proposition 1,5]{OnPe}\label{OnPe}
    $\theta(x)$ is stable as a partial type if and only if the definable set $\theta(\U)$ is stable and stably embedded.
\end{fact}
We now recall several notions of the stable part of a structure.  
\begin{definition}\cite[Definition~3.1]{HHM}.\label{def: stable part of the structure} Let $\A=\acl^{eq}(\A)$ be a small set of imaginary parameters. We denote by $\St_{\A}$  the multi-sorted structure $(\mathrm{D}_{i}, \mathrm{R}_{j})_{i \in \I, j \in \mathrm{J}}$ whose sorts $\mathrm{D}_{i}$ are the $\A$-definable, stable, stably embedded subsets of $\M$. For each finite set of sorts $\mathrm{D}_i$, all the $\A$-definable relations on their finite product are included as $\emptyset$-definable relations $\mathrm{R}_{j}$. 
\end{definition}
The following is \cite[Lemma $3.2$]{HHM}. 
\begin{fact}\label{fact: stablepartproperties} The structure $\St_{\A}$ is stable, stably embedded over $\A$ and eliminates imaginaries. 
\end{fact}
\begin{notation}\label{not:notationstabletrace}
For every parameter set $B$, we put $\St_\A(B)
:=
\St_\A \cap\dcl^{eq}(A B).$\end{notation}
For a small model $M$, formulas with parameters in $\St_M(B)$ will play an important role in this paper. 


\subsection{Keisler measures}

Throughout this subsection, \(T\) is a complete first-order \(\mathcal L\)-theory and \(\U\models T\) is a sufficiently saturated monster model.

\begin{definition}
Let \(A\subseteq \U\) be a small set of parameters. A \emph{Keisler measure over \(A\)} in the variable \(x\) is a finitely additive probability measure on $\mathbb{B}(\A)$: the Boolean algebra of \(\mathcal L_x(A)\)-definable subsets of \(\U^x\). A \emph{global Keisler measure} is a Keisler measure over \(\U\).
\end{definition}

\begin{definition}\label{def:smooth}
Let \(M\models T\), and let \(\mu\) be a Keisler measure over \(M\).
\begin{enumerate}
    \item We say that \(\mu\) is \emph{smooth} if it has a unique extension to a Keisler measure over every model \(N\supseteq M\).
    \item A global Keisler measure \(\mu\) is said to be \emph{smooth over \(M\)} if \(\mu\rest_{\mathcal L_x(M)}\) is smooth.
\end{enumerate}
\end{definition}

In order to use measure theoretic results one needs to extend the additive Keisler measure on the Boolean algebra  of definable sets $\mathbb B_x(A)$ to $\sigma$-additive measures. There are two equivalent ways of doing this. In \cite{Keisler} Keisler identifies a formula with its realizations in the saturated model $\U$ and he proves in \cite[Theorem 1.2]{Keisler} that under this identification one can extend any Keisler measure to a countably additive measure on the $\sigma$-algebra generated by the definable sets in $\U^{|x|}$. The more modern (equivalent) approach is to work in the space of types $S_x(A)$ (or $S_x(\U)$ if we want to vary the parameter set $A$), identify any formula with the set of types containing it, and again extending any Keisler measure to the $\sigma$-algebra generated by definable sets $\mathbb B_x(A)$ in the space $S_x(A)$ (or $S_x(\U)$, respectively). This paper owes a lot to Keisler's ideas and because of this we will mention Keisler's approach. However, some of the definitions and results we use are from more recent articles (such as \emph{Borel definability} in the following definition), which are based in the type-space view.   

\begin{definition}\label{def:invariant-definable-fs}
Let \(\A\subseteq \mathrm{B}\subseteq \U\) be small sets of parameters. Let \(\mu\) be a Keisler measure over $\mathrm{B}$.
\begin{enumerate}
    \item We say that \(\mu\) is \emph{\(\A\)-invariant} if, for every formula \(\phi(x,y)\) and all tuples \(b,b'\in \mathrm{B}^y\), whenever
    $b\equiv_{\A} b',$
    we have
    $\mu(\phi(x,b))=\mu(\phi(x,b')).$

    \item We say that \(\mu\) is \emph{finitely satisfiable in \(\A\)} if, for every formula \(\phi(x,y)\) and every \(b\in \mathrm{B}^y\), whenever
    $ \mu(\phi(x,b))>0,$
    there is some \(a\in \A^x\) such that
    $ \U\models \phi(a,b).$

    \item We say that $\mu$ is \emph{nonforking over $\A$} if for any formula \(\phi(x,y)\) and every \(b\in \mathrm{B}^y\) if $\phi(x,b)$ forks over $\A$ then $\mu(\phi(x,b))=0$.

    \item We say that \(\mu\) is \emph{definable over \(\A\)} if \(\mu\) is \(\A\)-invariant and for every formula \(\phi(x,y)\) and every \(r\in [0,1]\), the set $\{q\in S_y(\A) \ | \  \mu(\phi(x,b))<r \text{ for any $b\in \mathrm{B}$ with $b\models q$}\}$ is open in \(S_y(\A)\).

\item We say that \(\mu\) is \emph{Borel definable over \(\A\)} if \(\mu\) is \(\A\)-invariant and for every formula \(\phi(x,y)\) and every \(r\in [0,1]\), the set
\[\{q\in S_y(\A) \ | \  \mu(\phi(x,b))<r \text{ for any $b\in \mathrm{B}$ with $b\models q$}\}, \ \text{is Borel in \(S_y(\A)\)}.\]

\item We write $\mu\rest_{\A}$ to denote the restriction of $\mu$ to $\mathbb{B}(\A)$: the boolean algebra of $\A$-definable subsets. We say that $\mu \rest_{\A}$ is \emph{stationary} if it has a unique non-forking global extension.  
\end{enumerate}
\end{definition}

\begin{fact}
\label{fact:support-keisler-measure}
Let $\A$ be a small set of parameters. Let $\mu$ be a Keisler measure over
$\A$. Define
\[
\operatorname{Supp}(\mu)
:=
\left\{
p\in S_x(\A) \ | \ 
\mu(\theta(x))>0
\text{ for every }\theta(x)\in p
\right\}.\]
Then $\operatorname{Supp}(\mu)$ is a closed, and hence compact,
subset of $S_x(\A)$. Moreover, for every formula $\theta(x)\in\mathcal L_x(\A)$,
$\mu(\theta(x))>0$ if and only if $[\theta(x)]\cap\operatorname{Supp}(\mu)\neq\varnothing.$
\end{fact}

\begin{proof}
This is \cite[Section~7.1, p.~99]{simon}, we include details for sake of completeness. A type $p\in S_x(\A)$ does not belong to
$\operatorname{Supp}(\mu)$ if and only if there is a formula
$\theta(x)\in p$ such that
$\mu(\theta(x))=0.$
Consequently,
$S_x(A)\setminus\operatorname{Supp}(\mu)=\bigcup_{\mu(\theta)=0}[\theta],$
where the union ranges over all $\mathcal L_x(A)$-formulas of
$\mu$-measure zero. Each set $[\theta]$ is clopen in the Stone
topology, so the complement of $\operatorname{Supp}(\mu)$ is open.
Thus $\operatorname{Supp}(\mu)$ is closed.  Since $S_x(A)$ is compact,
the support is compact.\\
For the second part of the statement, if $p\in[\theta]\cap\operatorname{Supp}(\mu),$
then $\theta\in p$, and the definition of the support gives
$\mu(\theta)>0.$ Conversely, suppose that
$\mu(\theta)>0$
and, toward a contradiction, that
$[\theta]\cap\operatorname{Supp}(\mu)=\varnothing.$
For every $p\in[\theta]$, choose a formula $\psi_p(x)\in p$ such
that
$\mu(\psi_p(x))=0.$
The clopen sets $[\psi_p]$, as $p$ ranges over $[\theta]$, form an
open cover of the compact set $[\theta]$. Hence there are
$p_1,\ldots,p_n\in[\theta]$ such that
$
[\theta]
\subseteq
\bigcup_{i=1}^n[\psi_{p_i}].
$
Equivalently,
$
\models
\forall x\,
\left(
\theta(x)\rightarrow
\bigvee_{i=1}^n\psi_{p_i}(x)
\right).
$
By monotonicity and finite subadditivity,
$
\mu(\theta)
\leq
\sum_{i=1}^n\mu(\psi_{p_i})
=
0,
$
contradicting $\mu(\theta)>0$. Therefore
$
[\theta]\cap\operatorname{Supp}(\mu)\neq\varnothing.
$
\end{proof}
The following is \cite[Theorem~3.2]{genmeasures}.
\begin{theorem}\label{thm:genstable}
Let \(\A\subseteq \U\) be small, and let \(\mu\) be a global \(A\)-invariant Keisler measure. Then the following are equivalent:
\begin{enumerate}
    \item \(\mu\) is definable over \(A\) and finitely satisfiable in some small model containing \(A\);
    \item \(\mu^\omega_{(x_0,x_1,\dots)}\rest_A\) is totally indiscernible;
    \item for every global \(A\)-invariant Keisler measure \(\lambda_y\), one has
    $\mu_x\otimes \lambda_y=\lambda_y\otimes \mu_x;$
    \item \(\mu\) commutes with itself, namely
    $\mu_x\otimes \mu_y=\mu_y\otimes \mu_x;$
    \item for some small model \(M_0\supseteq A\), for every \(X\subseteq S_x(A)\) Borel  set  over $\A $, and every formula \(\phi(x)\) over \(\U\), if
    $ \mu(X\cap \phi(x))>0,$
    then there is \(a\in M_0^x\) such that \(a\in X\) and
    $\U\models \phi(a).$
\end{enumerate}
\end{theorem}
\begin{definition}\label{def:generically-stable-global-measure}
Let \(\A\subseteq \U\) be small set of parameters. Let \(\mu\) be a global \(\A\)-invariant Keisler measure. We say that \(\mu\) is \emph{generically stable over \(A\)} if it satisfies one, equivalently all, of the conditions of Theorem~\ref{thm:genstable}.
\end{definition}

\begin{fact}\label{fact:global-invariant-extension}
Assume that $T$ is $\mathrm{NIP}$. Every Keisler measure over a model $M$ admits
a global $M$-invariant extension.
\end{fact}

\begin{proof}
A Keisler measure over $M$ does not fork over $M$. Hence the result is
\cite[Lemma~5.2]{genmeasures}.
\end{proof}
\begin{fact}\label{fact:smoothmeasurezero} Let $\lambda$ be a smooth Keisler measure over $M$.
And let $p$ be a non-realized type over $M$. Then $\widetilde{\lambda}(p)=0$, where $\widetilde{\lambda}$ is the Borel measure on space of types over $M$ induced by $\lambda$. 
\end{fact}
\begin{proof}
By \cite[Lemma $4.1$]{SimonGenericallyStable} the localization at $\{p\}$, namely the Dirac measure
$\delta_p$, is smooth over $M$.
If $p$ is not realized in $M$, it has distinct
realizations $a,a'\in\mathbb U$.
The Dirac measures $\delta_a$ and $\delta_{a'}$ give distinct
global extensions of $\delta_p$, but then $\delta_{p}$ is not smooth.
\end{proof}

\begin{notation}
For a fixed tuple of variables $x$ and a parameter set $\A\subseteq\U$, we will use $\mathbb B_x(\A)$ to denote the Boolean algebra of $\mathcal L_x(\A)$-formulas modulo
logical equivalence in $\U$. We identify a formula with its
equivalence class in this Boolean algebra. When the tuple of variables
is clear, we simply write $\mathbb B(\A)$.

For any Boolean sub-algebra $\mathbb F\subseteq\mathbb B_x(\U)$  and
any global Keisler measure $\mu$ we write $\mu\rest_{\mathbb F}$
for the restriction of $\mu$ to $\mathbb F$. For any small model $M$ of $\U$ we will write $\mu\rest_{M}$ to indicate the restriction of $\mu$ to $\mathcal{L}_{x}(M)$. 
\end{notation}
\begin{definition}
\label{def:F-approximation}
Let $\mu$ be a global Keisler measure and let
$\mathbb F\subseteq\mathbb B_x(\U)$ be a Boolean subalgebra.

\begin{enumerate}
\item Let $\phi(x,b)$ be an $\mathcal L_x(\U)$-formula. We say that
$\phi(x,b)$  is
\emph{$\mathbb F$-approximated with respect to $\mu$} if, for every
$\epsilon>0$, there are formulas
$\theta_-(x),\theta_+(x)\in\mathbb F$
such that
$\U\models
\forall x\,
\bigl(
\theta_-(x)\rightarrow
\phi(x,b)\rightarrow
\theta_+(x)
\bigr)$
and
$|\mu(\theta_+(x))-\mu(\theta_-(x))|<\epsilon.$

\item We say that $\mu$ is \emph{$\mathbb F$-approximated} if every
$\mathcal L_x(\U)$-formula is $\mathbb F$-approximated with respect
to $\mu$.

\item For an $\mathcal L_x(\U)$-formula $\phi(x,b)$, we define its
\emph{inner $\mathbb F$-measure} and
\emph{outer $\mathbb F$-measure} by
\begin{align*}
\mu_{\mathrm{in}}^{\mathbb F}(\phi(x,b))
&:=
\sup
\left\{
\mu(\theta(x)):
\begin{array}{l}
\theta(x)\in\mathbb F,\\
\U\models
\forall x\,
\bigl(
\theta(x)\rightarrow\phi(x,b)
\bigr)
\end{array}
\right\},\\
\mu_{\mathrm{out}}^{\mathbb F}(\phi(x,b))
&:=
\inf
\left\{
\mu(\theta(x)):
\begin{array}{l}
\theta(x)\in\mathbb F,\\
\U\models
\forall x\,
\bigl(
\phi(x,b)\rightarrow\theta(x)
\bigr)
\end{array}
\right\}.
\end{align*}
\end{enumerate}
\end{definition}
The following lemma will play a fundamental role in several of our arguments. It is quite close to Keisler's \cite[Lemma 1.3]{Keisler} the result there is not in the precise form needed here. For this reason, and for the sake of completeness, we include a detailed proof.
\begin{lemma}
\label{lem:unique-extension-criterion}
Let $\mu$ be a global Keisler measure, let
$\mathbb F\subseteq\mathbb B_x(\U)$ be a Boolean subalgebra, and let
$\phi(x,b)$ be an $\mathcal L_x(\U)$-formula. The following are
equivalent.

\begin{enumerate}
\item The formula $\phi(x,b)$ is $\mathbb F$-approximated with
respect to $\mu$.

\item One has
$\mu_{\mathrm{in}}^{\mathbb F}(\phi(x,b))=\mu_{\mathrm{out}}^{\mathbb F}(\phi(x,b)).$

\item The value of $\phi(x,b)$ is uniquely determined by
$\mu\rest_{\mathbb F}$: whenever $\nu_{1}$ and $\nu_{2}$ are  global Keisler measures satisfying $\nu_{1}\rest_{\mathbb F}=\mu\rest_{\mathbb F}=\nu_{2} \rest_{\mathbb{F}}$ one has $\nu(\phi(x,b))=\mu(\phi(x,b)).$
\end{enumerate}
\end{lemma}

\begin{proof}
Let $\nu$ be any
Keisler measure defined on a Boolean algebra containing both
$\mathbb F$ and $\phi(x,b)$, and suppose that
$\nu\rest_{\mathbb F}=\mu\rest_{\mathbb F}.$
If $\theta(x)\in\mathbb F$ and
$\U\models
\forall x\,
\bigl(
\theta(x)\rightarrow\phi(x,b)
\bigr),$
then monotonicity of $\nu$ gives
$\mu(\theta(x))=\nu(\theta(x))\leq\nu(\phi(x,b)).$
Taking the supremum over all such $\theta$ yields
$\mu_{\mathrm{in}}^{\mathbb F}(\phi(x,b))\leq\nu(\phi(x,b)).$
Similarly, if $\theta(x)\in\mathbb F$ and
$\U\models
\forall x\,
\bigl(
\phi(x,b)\rightarrow\theta(x)
\bigr),$
then
$\nu(\phi(x,b))\leq\nu(\theta(x))=\mu(\theta(x)).$
Taking the infimum gives
$\nu(\phi(x,b))\leq \mu_{\mathrm{out}}^{\mathbb F}(\phi(x,b)).$
Thus every extension satisfies
\[\mu_{\mathrm{in}}^{\mathbb F}(\phi(x,b))
\leq
\nu(\phi(x,b))
\leq
\mu_{\mathrm{out}}^{\mathbb F}(\phi(x,b)).
\tag{$\star$}\]
We now prove the equivalences.\\
\noindent
\emph{$(1)\rightarrow(2)$.}
Assume that $\phi(x,b)$ is $\mathbb F$-approximated with respect to $\mu$. Let
$\epsilon>0$. Choose
$
\theta_-(x),\theta_+(x)\in\mathbb F
$
such that
$\theta_-(x)\rightarrow\phi(x,b)\rightarrow\theta_+(x)$
and $\mu(\theta_+(x))-\mu(\theta_-(x))<\epsilon.$
By the definitions of inner and outer measure,
$\mu_{\mathrm{in}}^{\mathbb F}(\phi(x,b))\geq\mu(\theta_-(x))$ and $\mu_{\mathrm{out}}^{\mathbb F}(\phi(x,b))\leq\mu(\theta_+(x)).$
Consequently,
$0
\leq
\mu_{\mathrm{out}}^{\mathbb F}(\phi(x,b))-\mu_{\mathrm{in}}^{\mathbb F}(\phi(x,b)) \leq
\mu(\theta_+(x))-\mu(\theta_-(x))<
\epsilon.$
Since $\epsilon>0$ was arbitrary, we obtain
$\mu_{\mathrm{in}}^{\mathbb F}(\phi(x,b))=\mu_{\mathrm{out}}^{\mathbb F}(\phi(x,b)).$

\smallskip

\noindent
\emph{$(2)\rightarrow(1)$.}
Assume that $(2)$ holds and let $r=\mu_{\mathrm{in}}^{\mathbb F}(\phi(x,b))=\mu_{\mathrm{out}}^{\mathbb F}(\phi(x,b))$ Applying
\textup{$\star$} to the original measure $\mu$ gives
$\mu_{\mathrm{in}}^{\mathbb F}(\phi(x,b))\leq\mu(\phi(x,b))\leq \mu_{\mathrm{in}}^{\mathbb F}(\phi(x,b)).$
Hence $\mu(\phi(x,b))=r.$ Let $\epsilon>0$. By the definition of the supremum, there is
$\theta_-(x)\in\mathbb F$ such that
$\theta_-(x)\rightarrow\phi(x,b)$ and
$\mu(\theta_-(x))>r-\frac{\epsilon}{2}.$
Similarly, by the definition of the infimum, there is
$\theta_+(x)\in\mathbb F$ such that
$\phi(x,b)\rightarrow\theta_+(x)$
and
$\mu(\theta_+(x))<r+\frac{\epsilon}{2}.$\\
Therefore,
$\mu(\theta_+(x))-\mu(\theta_-(x))<\left(r+\frac{\epsilon}{2}\right)-\left(r-\frac{\epsilon}{2}\right)=\epsilon.$
Thus $\phi(x,b)$ is $\mathbb F$-approximated with respect to $\mu$.

\smallskip

\noindent
\emph{$(2)\rightarrow(3)$.}
Let  $r=\mu_{\mathrm{in}}^{\mathbb F}(\phi(x,b))=\mu_{\mathrm{out}}^{\mathbb F}(\phi(x,b))$. Let $\nu_1$ and $\nu_2$ be global
extensions of $\mu\rest_{\mathbb F}$. By \textup{($\star$)}, for
$i=1,2$,
$
r=\mu_{\mathrm{in}}^{\mathbb F}(\phi(x,b))
\leq
\nu_i(\phi(x,b))
\leq
\mu_{\mathrm{out}}^{\mathbb F}(\phi(x,b))=r.
$
Since $\mu_{\mathrm{in}}^{\mathbb F}(\phi(x,b))=\mu_{\mathrm{out}}^{\mathbb F}(\phi(x,b))$, both values must be equal to $r$. Hence
$\nu_1(\phi(x,b))=\nu_2(\phi(x,b)).$

\smallskip

\noindent
\emph{$(3)\rightarrow(2)$.}
Suppose, toward a contradiction, that
$
\mu_{\mathrm{in}}^{\mathbb F}(\phi(x,b))< \mu_{\mathrm{out}}^{\mathbb F}(\phi(x,b)).
$
Choose two distinct real numbers
$
r_1,r_2\in[\mu_{\mathrm{in}}^{\mathbb F}(\phi(x,b)),\mu_{\mathrm{out}}^{\mathbb F}(\phi(x,b))].
$
By \cite[Lemma~1.3.(iv)]{Keisler}, for each $i=1,2$ there is an extension
$\nu_i$ of $\mu\rest_{\mathbb F}$ to a global Keisler measure such
that
$
\nu_i(\phi(x,b))=r_i.
$
Thus
$
\nu_1(\phi(x,b))
\neq
\nu_2(\phi(x,b)),
$
contradicting $(3)$. Therefore $\mu_{\mathrm{in}}^{\mathbb F}(\phi(x,b))=\mu_{\mathrm{out}}^{\mathbb F}(\phi(x,b))$.
\end{proof}

\begin{proposition}
\label{prop: inner and outer measures}
Let $\mu$ be a global measure, let $A\subset \U$ be a small set, let
$\mathbb F\subseteq\mathbb B_x(\U)$ be a Boolean subalgebra containing $\mathbb B_x(A)$, and let $\mu^{\mathbb F}$ be the restriction of $\mu$ to $\mathbb F$. 

Then the following hold.

\begin{enumerate}
\item If $\mu^{\mathbb F}$ is $A$-definable, then $\mu_{\mathrm{out}}^{\mathbb F}$ is $A$-definable.

\item If $\mu^{\mathbb F}$ is finitely satisfiable over $A$, then $\mu_{\mathrm{inn}}^{\mathbb F}$ is finitely satisfiable over $A$.

\item If $\mu$ is $\mathbb F$-approximated, then $\mu$ is $A$-definable/finitely satisfiable over $A$ if and only if $\mu^{\mathbb F}$ is $A$-definable/finitely satisfiable over $A$.\end{enumerate}
\end{proposition}

\begin{proof} 
The third item follows from the first two since \cref{lem:unique-extension-criterion} implies that if $\mu$ is $\mathbb F$-approximated then $\mu$ is equal to both $\mu_{\mathrm{inn}}^{\mathbb F}$ and $\mu_{\mathrm{out}}^{\mathbb F}$ on any definable set.

We now prove $(1)$. Let $\phi(x,w)$ be a $\mathcal{L}$-formula, $r \in [0,1]$ and let 
\[O=\{ q \in S_{y}(A) \ | \ \mu_{\mathrm{out}}^{\mathbb F}(\phi(x,b'))< r  \ \text{for all $b' \models q$}\}.\] We want to show that $O$ is an open set. Fix an element $b \in \mathbb{M}$ and let $q_{0}=\tp(b/A) \in O$. By definition of the outer measure, there is a formula $\theta_{+}(x,d) \in \mathbb F$ such that $\phi(x,b) \rightarrow \theta_{+}(x,d)$ and  $\mu^\mathbb F(\theta_{+}(x,d)) < r$. By hypothesis there is a formula $\eta(z) \in tp(d/A)$ such that for any $\mathbb{M} \models \eta(d')$ we have that $\mu^\mathbb F\big( \theta_+(x,d')\big) < r$.\\ 

Let $\phi(w):=\exists z \big( \eta(z) \wedge \forall x\big( \phi(x,w)\rightarrow \theta_+(x,z)\big).$
Note that $\phi(w) \in \tp(b/A)$ and for any $b'$ such that $\mathbb{M} \models \phi(b')$ then there is $\mathbb{M}\models \eta(d')$ such that $\phi(x,b')\rightarrow \theta_+(x, d')$ and $\mu^\mathbb F(\theta_+(x,d'))<r$. By definition,  $\mu_{\mathrm{out}}^{\mathbb F}\big(\phi(x,b')\big)<r$. Hence $q_{0} \in [\phi(w)] \subseteq O$, thus $O$ is open.

For $(2)$, we want to show that if $\mu^{\mathbb{F}}$ is  is finitely satisfiable over $M$, so it is $\mu^\mathbb F_{\mathrm{inn}}$. Assume that $\mu_{\mathrm{inn}}^{\mathbb F}(\phi(x,b))>0$ for a fixed parameter $b \in \mathbb{M}$. By definition of the inner measure there is a formula $\theta_{-}(x,d) \in \mathbb F$ such that $\mathbb{M} \models \forall x \big( \theta_{-}(x,d) \rightarrow \phi(x,b))$ and $\mu^\mathbb F(\theta_{-}(x,d))>0$. Since $\mu^\mathbb F(\theta_{-}(x,d))$ was finitely satisfiable there is some $m \in M$ such that $\mathbb{\U} \models \theta_{-}(m,d)$. Hence $\mathbb{\U} \models \phi(m,b)$, as required. \end{proof}

We conclude this subsection with the folklore result on uniqueness theorem for finite measures due to Dynkin’s Theorem (see \cite[Theorem $3.2$-$3.3$]{billingsley2012}).
\begin{fact}\label{fact:uniquenessmeasures} Let $(X, \mathcal{F})$ be a measurable space and $\mathcal{P} \subseteq \mathcal{F}$. And assume that:
\begin{enumerate}
\item the set $X \in \mathcal{P}$,
\item $\mathcal{P}$ is closed under finite intersections,
\item the $\sigma$-algebra generated by $\mathcal{P}$ is $\mathcal{F}$.
\end{enumerate}
If $\mu$ and $\nu$ are finitely countably additive measures on $\mathcal{F}$ such that $\mu(A)=\nu(A)$ for all $A \in \mathcal{P}$ then $\mu(B)=\nu(B)$ for all $B \in \mathcal{F}$. 
\end{fact}

\subsection{Some results in \texorpdfstring{$\mathrm{ACVF}$}{ACVF}}

\begin{fact}\label{fact:acvf-stably-embedded-sorts}
In \(\ACVF\), the residue field \(\kres\) is a stably embedded pure algebraically closed field and eliminates imaginaries. The value group \(\Gamma\) is a stably embedded pure divisible ordered abelian group and eliminates imaginaries. In particular, \(\Gamma\) is distal.
\end{fact}

\begin{proof}
The stable embeddedness and induced structures on \(\kres\) and \(\Gamma\) are standard consequences of quantifier elimination for \(\ACVF\); see \cite[Proposition $7.1$ and $7.2$]{HHM}. The residue field is a model of \(\mathrm{ACF}\), and the value group is a model of the theory of divisible ordered abelian groups. The value group is distal because divisible ordered abelian groups are o-minimal; see \cite[Corollary $2.30$]{distal}.
\end{proof}

\begin{notation}\label{not:residue-value-prodefinable-maps}
Let \(M\models \ACVF\) be a spherically complete model. For a tuple \(a\) from the valued-field sort, we write
$\mathbf f(a)$
for a pro-definable enumeration of the residue-field part
$\kres(Ma)\colon=\kres(\mathbb{M}) \cap \dcl^{eq}(Ma)$
and
$\mathbf g(a)$
for a pro-definable enumeration of the value-group part
$\Gamma(Ma)\colon=\Gamma(\mathbb{M}) \cap \dcl^{eq}(Ma).$
Equivalently, \(\mathbf f\) and \(\mathbf g\) are \(M\)-pro-definable maps which record, respectively, the residue-field and value-group data generated by \(a\) over \(M\). If \(p\) is a global \(M\)-invariant type, we write
$p^{st}:=\mathbf f_*p$
for its residue-field projection, and
$p^\Gamma:=\mathbf g_*p$
for its value-group projection. Note that both $p^{st}$ and $p^{\Gamma}$ are still $M$-invariant.
\end{notation}

\begin{fact}\label{fact:implication-from-residue-value-restrictions} Let $M \models \mathrm{ACVF}$ be a spherically complete model. 
Let \(p\) be a global \(M\)-invariant type in the valued-field sort, and let \(b\in K(\U)\). Let \(\mathbf f\) and \(\mathbf g\) denote the residue-field and value-group projection maps from \cref{not:residue-value-prodefinable-maps}. If \(a\models p\rest_M\), then
$a\models p\rest_{Mb}$
if and only if
$\mathbf f(a)\models p^{st}\rest_{Mb}$ and $
\mathbf g(a)\models p^\Gamma\rest_{Mb}.$
Equivalently, $\p\rest_{M} \cup \p^{st} \rest_{Mb} \cup \p^{\Gamma}\rest_{Mb} \vdash \p \rest_{Mb}$. 
\end{fact}
\begin{proof}
The direction $a \models \p \rest_{M b}$ implies that $\mathbf{f}(a) \models \p^{st}\rest_{Mb}$ and $\mathbf{g}(a) \models \p^{\Gamma} \rest_{Mb}$ is immediate. For the converse, assume $\mathbf{f}(a)=\mathbf{k}(Ma) \models \p^{st} \rest_{Mb}$. In particular, $\mathbf{k}(Ma)\downfree_{M} Mb$ hence $\mathbf{k}(Ma) \downfree_{\mathbf{k}(M)} \mathbf{k}(Mb)$ (\emph{c.f.} \cite[Lemma $2.6$]{orthogonal}). Hence, $\mathbf{k}(Ma)$ and $ \mathbf{k}(Mb)$ are algebraically independent over $\mathbf{k}(M)$.  Since $\mathbf{k}(M)$ is algebraically closed, the residue field $\mathbf{k}(Ma)$ is linearly disjoint from $\mathbf{k}(Mb)$ over $\mathbf{k}(M)$. Similarly,  $\mathbf{g}(a) \models \p^{\Gamma} \rest_{Mb}$ implies $\Gamma(Ma) \downfree_{\Gamma(M)} \Gamma(Mb)$. In particular, $\Gamma(Ma) \cap \Gamma(Mb)=\Gamma(M)$. (\emph{e.g.} by \cite[Remark $3.1$ and Fact $2.5$]{Vic-ResDom}). By \cite[Corollary 12.11]{HHM}.(i) (or \cite[Proposition 3.3]{Vic-ResDom}), 
\[\tp(a/M) \cup \tp(\mathbf{k}(Ma)/ \mathbf{k}(M b)) \cup \tp(\Gamma(Ma)/ \Gamma(Mb)) \vdash \tp(a/Mb)\]
Hence $\p\rest_{M} \cup \p^{st} \rest_{Mb} \cup \p^{\Gamma}\rest_{Mb} \vdash \p \rest_{Mb}$. \\
\end{proof}
We will apply in particular the following lemma for $T=\mathrm{ACVF}$ and $\Gamma$ the value group, which is stably embedded by \cref{fact:acvf-stably-embedded-sorts}.

\begin{lemma}
\label{lem:finite-gamma-normal-form} Let $T$ be $\mathrm{NIP}$ and $M \models T$. Let $\Gamma$ be a $\mathcal{L}_{\Gamma}$-structure. Assume that $\Gamma$ is a $\emptyset$-definable set in $M$ that is stably embedded, eliminates imaginaries and its induced structure is the $\mathcal{L}_{\Gamma}$-structure. Let $X$ be an $M$-definable set, and let
$\varphi(x;y)$ be an $\mathcal L(M)$-formula, where $x$ ranges over $X$ and $y$ ranges over $\Gamma^m$. Then there exist an $M$-definable set $S\subseteq\Gamma^r,$ an $M$-definable map $g:X\mapsto S,$
and a $\mathcal{L}_{\Gamma}$-formula
$\theta(u;y)$, with
parameters from $\Gamma(M)$, such that
$T\models
\forall x\in X\,\forall y\in\Gamma^m\,
\bigl(
\varphi(x;y)
\leftrightarrow
\theta(g(x);y)
\bigr).$
\end{lemma}

\begin{proof}
Consider the $M$-definable equivalence relation on $X$ given by
$a\sim a'$ if and only if $\forall y\in\Gamma^m\,
\bigl(
\varphi(a;y)\leftrightarrow\varphi(a';y)
\bigr).$ 
Thus two elements of $X$ are equivalent precisely when they define
the same $\phi$-subset of $\Gamma^m$.
Since $\Gamma$ is stably embedded  and it eliminates imaginaries, there is a $\mathcal{L}_{\Gamma}$-formula $\theta(y;u)$, depending only on $\varphi$, such that
each set $\varphi(a;\Gamma^m)$ is of the form $\theta(\Gamma^m;c_{a})$ for some tuple $c_{a}$ from $\Gamma \cap \dcl^{eq}(Ma)$. Indeed, elimination of imaginaries allows one to take the
canonical parameter of the set $\Theta(\Gamma^m;c)$  in a $M$-definable subset $S$ of a finite Cartesian power of $\Gamma$. By compactness, we can make such coice of canonical parameters uniform, and we obtain a map $g\colon X \to S$ where each element $a$ is sent to the canonical parameter $c_{a}$, such function clearly factors through the quotient $X/\sim$. By construction, $g(a)$ is the canonical parameter of the subset $\Theta(g(a);\Gamma^{m})$, which coincides with 
$\varphi(a;\Gamma^m)$. Therefore
$\varphi(a;y)\leftrightarrow\theta(g(a);y)$
for every $a\in X$ and $y\in\Gamma^m$, which proves the main statement. 
\end{proof}
We will be interested in particular in the stable part in $\mathrm{ACVF}$, a more algebraic description of this part of the structure is given in \cite[Propositions~7.7 and~7.8]{HHM}. 

\section{Stably dominated measures and orthogonality}
In this section we give definitions that generalize the notions of stable domination and orthogonality to a stably embedded sort for global Keisler measures. We state the definitions in an arbitrary theory, and state the first theorem characterizing measures in $\mathrm{ACVF}$. Through this section we work with $T$ an arbitrary $\mathcal{L}$-first order theory, $\A=\acl^{eq}(\A)$ and $B$ denote small sets of parameters, $M$ is a small model and $\mathbb{M}$ is the monster model.

\begin{definition}\label{def:stablebooleanalgebra} Let $B \supseteq A=\acl^{eq}(A)$ be a small set of parameters.
\begin{enumerate}
\item A \emph{special pushforward stable formula} is a formula of the form: $\psi(x,a) \wedge \alpha(f(x),d)$, where $\psi(x,a)$ is a $\mathcal{L}(A)$-formula, $f: \psi(\M,a)\to D$ is an $A$-definable map, $D$ is a finite product of sorts in $\mathrm{St}_{A}$ and $\alpha(u,d)$ is a formula in the stably embedded structure $\mathrm{St}_{A}$ where $d \in \mathrm{St}_{A}(B)$.
\item The \emph{stable Boolean algebra over $A$ with parameters from $B$,} denoted $\defst_A(B),$ is the Boolean sub-algebra of $\mathbb B_x(B)$ generated by special pushforward stable formulas. 
\item If $\mu$ is a global Keisler measure, we write
$\mu_{\mathrm{st},A}
:=
\mu\rest_{\defst_A(\U)}.$ When the base is fixed, we simply write $\mu_{\mathrm{st}}$.
\end{enumerate}
\end{definition}

Notice that by stable embeddedness of $\St_A$ (\emph{c.f.} \cref{fact: stablepartproperties}), we could have allowed $d$ to be any tuple from $B$. So (in Keisler's view) $\defst_A(B)$ is the expansion 
of $\mathbb B_x(A)$ to formulas with parameters in $\mathrm{St}_{A}(B)\colon =\mathrm{St}_{A} \cap \dcl^{eq}(B)$ where the variable belongs to a finite product of sorts of $\St_A$.\\
We now specialize the abstract approximation formalism introduced in
Definition~\ref{def:F-approximation} to the stable Boolean algebra.

\begin{definition}
\label{def:stable-approximation-and-domination}
Let $\mu$ be a measure on $\mathbb B_x(B)$ for $B\supseteq A$.

\begin{enumerate}
\item We say that $\mu$ is \emph{stably approximated over $A$} if
$\mu$ is $\defst_A(B)$-approximated in the sense of
Definition~\ref{def:F-approximation}.

\item If $\mu$ is a global Keisler measure, we say that $\mu$ is \emph{stably dominated over $A$} if
$\mu_{\mathrm{st},A}$ has a unique extension to the full Boolean
algebra $\mathbb B_x(\U)$, and this unique extension is $\mu$.
\end{enumerate}
Note that the notions in $(1)$ and $(2)$ are equivalent by \cref{lem:unique-extension-criterion}. 
\end{definition}
\begin{definition}\label{def:thestablepushforward} Let $\mu$ be a global Keisler measure and $A=\acl^{eq}(A)$ be a small algebraically closed base of imaginary parameters. Let $\psi(x,a)$ be a $\mathcal{L}(A)$-formula and let $f$ be an $A$-definable map  $f: \psi(x,a) \rightarrow D$, where $D$ is a finite product of sorts in $\mathrm{St}_{A}$. Let $\nu$ be the localized measure of $\mu$ on the definable set $\psi(x,a)$. We define the \emph{stable push-forward of $\mu$ on $\psi(x,a)$ along $f$} to be the pushforward measure of $\nu$ along $f$, which we denote by $\mu_{f}$. More explicitly, for $\alpha(u,d)$ a formula within the stable structure $\mathrm{St}_{M}(\mathbb{M})$, where $d \in \mathrm{St}_{M}(\mathbb{M})$ and $u$ is a variable varying in $f(\psi(x,a)) \subseteq \mathrm{St}_{M}(\mathbb{M})$, we define:
\[ \mu_{f} \big( \alpha(u,d)\big)= \nu\big(\alpha(f(x),d)\big)=\mu \big( \psi(x,a) \wedge \alpha(f(x),d)\big).\] 
\end{definition}
\begin{remark}\label{rmk:thestablepushforward} The stable pushforward of $\mu$ on $\psi(x,a)$ along $f$ defined in \cref{def:thestablepushforward} is a Keisler measure on a stable and stably embedded $A$-definable set $D$, centered on the definable subsets of $f(\psi(x,a)) \subseteq D$. Since $\mu$ is $A$-invariant, $\psi(x,a)$ is a $\mathcal{L}(A)$-formula and $f$ is $A$-definable then the pushforward of $\mu$ on $\psi(x,a)$ along $f$ is also $A$-invariant. In particular it is generically stable over $A$ so it is definable and finitely satisfiable over $A$ (\emph{c.f.} \cref{thm:genstable}). 
\end{remark}

\begin{lemma}\label{lem: babycontinuity} Let $A=\acl^{eq}(A)$ and  $\mu$ be a global $A$-invariant Keisler measure. Let $\psi(x,a)$ be a $\mathcal{L}(A)$-formula, $f: \psi(x,a) \to D$ where $D$ is a finite product of sorts in $\mathrm{St}_{A}$. Let $\alpha(u,z)$ be a $\mathcal{L}(A)$-formula in the stable structure $\mathrm{St}_{A}(\mathbb{M})$. 
Then:
\begin{enumerate}
\item for every $r \in [0,1]$, then set \[O:=\{ q \in S_{yz}(A) \ | \ \mu(\psi(x,a) \wedge \alpha(f(x),d')) < r \ \text{for any $(a,d')\in \mathbb{M}$ with $(a,d') \models q$} \}\]
is open. 
\item Furthermore, if the base $A$ is a model $M$, then given $d \in \mathrm{St}_{M}(\mathbb{M})$ a fixed parameter, if $\mu(\psi(x,a) \wedge \alpha(f(x),d))>0$ then there is $m \in M$ such that $\mathbb{M} \models \psi(m,a) \wedge \alpha(f(m),d)$.
\end{enumerate}
\end{lemma}
\begin{proof}Let $q_{0}=\tp(a,d/M)$, where $d \in \mathrm{St}_{A}(\mathbb{M})$. Let $\mu_{f}$ be the pushfoward measure of $\mu$ on $\psi(x,a)$ along $f$, as in \cref{def:thestablepushforward}. By \cref{rmk:thestablepushforward} such pushforward is $A$-invariant and generically stable over $A$, so it is $A$-definable. In particular, for a fixed $r \in [0,1]$ the set $\{ p \in S_{z}(A) \ |\ \mu_{f}(\alpha(u,d'))<r \ \text{for any $d' \models p$} \ \}$ is an open set, so there is a formula $\delta(z) \in \tp(d/A)$ such that for every $d' \in \mathrm{St}_{M}(\mathbb{M})$ we have  $\mu_{f}(\alpha(u,d'))<r$ whenever $\mathbb{M}\models \delta(d')$. Let $\eta(y,z)\colon = \big(y=a \wedge \delta(z)\big)$.  Then $\eta(y,z) \in q_{0}$ and by the definition of the stable pushforward $\mu\big( \psi(x,a) \wedge \alpha(f(x),d')\big)=\mu_{f}\big(\alpha(u,d')\big)< r$ for any $\mathbb{M} \models \eta(a',d')$. 
Thus $q_{0} \in [\eta(y,z)]$ and $O$ is an open set. \\

For the second part of the statement, note that by definition of the stable pushforward of $\mu$ on $\psi(x,a)$ along $f$, we have that $\mu_{f}(\alpha(u,d))=\mu(\psi(x,a) \wedge \alpha(f(x),d))>0$. By \cref{rmk:thestablepushforward} the stable pushforward of $\mu$ on $\psi(x,a)$ along $f$ is finitely satisfiable over $\mathrm{St}_{M}(M)\preceq \mathrm{St}_{M}(\mathbb{M})$). Hence, there is some $n_{0} \in \mathrm{St}_{M}(M)$ such that $\mathbb{M}\models \alpha(n_{0},d)$.  Since $f: \psi(x,a) \rightarrow f(\psi(x,a)) \subseteq D \subseteq \mathrm{St}_{A}(\mathbb{M})$, and $\alpha(u,d)$ is a definable subset of $f(\psi(x,a))$, we have that $f$ is surjective. Take $m_{0} \in M$ be such that $f(m_{0})=n_{0}$, hence $\mathbb{M} \models \psi(m_{0},a) \wedge \alpha(f(m_{0}),d)$, as required. 
\end{proof}

\begin{proposition}\label{prop:stableapproxthengenstable} Let $T$ be a $\mathrm{NIP}$-theory. Let $M$ be a small model and $\mu$ a global $M$-invariant Keisler measure. If $\mu$ is stably dominated over $M$ then $\mu$ is generically stable. 
\end{proposition}

\begin{proof}
By \cref{prop: inner and outer measures}, the result follows from \cref{lem: babycontinuity}. 
\end{proof}

\begin{definition}\label{def:measure-orthogonal-gamma}
Let $\mu$ be a global $A$-invariant Keisler measure, and let $\Gamma$ be a stably embedded set. We say that
$\mu$ is \emph{orthogonal to $\Gamma$} if, for every finite
tuple
$
b\in\Gamma(\U)^n,
$
the measure $\mu\rest_A$ has a unique extension to $Ab$.
\end{definition}
We can now formula the statement of the first theorem. 
\begin{theorem}
\label{thm:main-generically-stable-measures-acvf}
Let $M$ be a spherically complete model of $\mathrm{ACVF}$, and let $\mu$ be a global $M$-invariant
Keisler measure. The following are equivalent:
\begin{enumerate}
\item $\mu$ is stably approximated over $M$/ $\mu$ is stably dominated over $M$,
\item $\mu$ is generically stable over $M$;
\item $\mu$ is orthogonal to the value group over $M$.
\end{enumerate}
\end{theorem}
\begin{remark} The implication $(1) \rightarrow (2)$ holds in any theory by \cref{prop:stableapproxthengenstable}. 

$(2) \rightarrow (3)$ will  hold for an arbitrary model of $\mathrm{ACVF}$. Indeed, its proof requires that the theory is $\mathrm{NIP}$ and that $\Gamma$ is a stably embedded set whose induced theory is distal. Our proof of the implication $(3)\rightarrow (1)$ only holds in $\mathrm{ACVF}$, and it does use spherically complete hypothesis on the base $M$ since it uses \cref{fact:implication-from-residue-value-restrictions}. 
\end{remark}

\subsection{Generically stable implies orthogonal to a distal stably embedded sort \texorpdfstring{$\Gamma$}{Gamma}.}
In this subsection we prove that given $\mu$ a generically stable measure in an $\mathrm{NIP}$ theory $T$ and $\Gamma$ is a stably embedded $\emptyset$-definable set, that eliminates imaginaries and whose induced structure is distal; then $\mu$ must be orthogonal to $\Gamma$. The following fact is the form of the distal strong honest-definition
theorem that we need. It follows closely, \cite[Theorem $9.21$]{simon} but we include the complete proof for sake of completeness.
\begin{lemma}
\label{fact:strong-definitions-gamma} Let $T$ be $\mathrm{NIP}$ and $M \models T$. Let $\Gamma$ be an $\emptyset$-definable set equipped with $\mathcal{L}_{\Gamma}$-structure, and assume that $\Gamma$ is stably embedded, eliminates imaginaries and its induced structure is the $\mathcal{L}_{\Gamma}$-structure which is distal.
Let $\phi(x;y)$ be an
$\mathcal L(M)$-formula in which $y$ ranges over $\Gamma^n$, and let $b\in\Gamma(\U)^n$. Then there are an elementary extension of pairs $(\U,M)\preceq(\U',M'),$ $\mathcal L(M)$-formulas $\psi_0(x;z)$ and $\psi_1(x;z)$, and a tuple
$d\in M'$ such that
$\U'\models
\forall x\,
\bigl(
\psi_0(x;d)\rightarrow
\phi(x;b)\rightarrow
\psi_1(x;d)
\bigr)$ and $\psi_0(M;d)
=\phi(M;b)=\psi_1(M;d).$
\end{lemma}

\begin{proof}
By \cref{lem:finite-gamma-normal-form} there are an
$M$-definable set
$S\subseteq\Gamma^m,$
an $M$-definable surjective map
$\rho:X \rightarrow S,$
and an $\mathcal L_\Gamma(\Gamma(M))$-formula $\chi(y;u)$, with
$u$ ranging over $S$, such that
$\U\models
\forall x\,\forall y\,
\bigl(
\phi(x;y)\leftrightarrow
\chi(y;\rho(x))
\bigr).$  For $a\in X(\U)$, the element $\rho(a)$ is the canonical
parameter of the subset
$\{c\in\Gamma(\U)^n | \ \U\models\phi(a;c)\}.$
Let
$A_\rho:=\rho(X(M))\subseteq S(\Gamma(\U)).$
Thus $A_\rho$ is the set of canonical parameters of the members of
the family defined by taking instances of $\phi$, whose parameters range on $M$.
Consider the structure
$\bigl(\Gamma(\U),A_\rho\bigr),$
namely the distal structure induced on $\Gamma(\U)$  expanded by a
predicate for the subset $A_\rho$.  The set $A_\rho$ is definable in the pair
$(\U,M)$, since
$
u\in A_\rho \leftrightarrow
\exists x\,
\bigl(
P(x)\wedge\rho(x)=u
\bigr),
$
where $P$ is the predicate naming the distinguished sub-model $M$ of $\mathbb{M}$. Consequently, the structure $(\Gamma(\U),A_{\rho})$ is interpretable in $(\U, M)$. \\
Choose a sufficiently saturated elementary extension of pairs $(\U,M)\preceq(\U',M').$ Let $A_\rho':=\rho(X(M')).$
By interpretability,
$\bigl(\Gamma(\U),A_\rho\bigr)
\preceq
\bigl(\Gamma(\U'),A_\rho'\bigr)$
is an elementary extension. We choose $(\U',M')$ sufficiently
saturated so $\bigl(\Gamma(\U'),A_\rho'\bigr)$ is also sufficiently saturated.  In particular, we can apply the argument in the proof \cite[Theorem~9.21]{simon} to the formula
$
\widetilde\phi(u;y)
:=
\chi(y;u),$
the subset $A_\rho\subseteq S(\Gamma(\U))$, and the tuple
$b\in\Gamma(\U)^n$, and we obtain an $\mathcal L_\Gamma(\Gamma(M))$-formula
$\Theta(y;z)$ and a tuple
$d\in A_\rho'\subseteq M'$
such that $\Gamma(\U')\models\Theta(b;d)$
and $\Theta(y;d)$ determines the
$\widetilde\phi^{\mathrm{opp}}$-type of $b$ over $A_\rho$. The last assertion states that for every $u\in A_\rho$ one has
$\Theta(y;d)\vdash\chi(y;u),$ if $\Gamma(\U)\models\chi(b;u),$ and 
$\Theta(y;d)\vdash\neg\chi(y;u),$ if $\Gamma(\U)\models\neg\chi(b;u).$ Given $a\in X(M)$ then $\rho(a)\in A_\rho$. In particular, for every $a\in X(M)$,
\begin{align}\label{eq:check}
\Theta(y;d)&\vdash\phi(a;y), \ \text{if} \ \U\models\phi(a;b), \ \text{and}\\
\Theta(y;d)&\vdash\neg\phi(a;y), \ \text{if} \ \U\models\neg\phi(a;b).
\end{align}
We now construct the required lower and upper approximations.\\
Let
$\psi_0(x;z)
:=\forall y\,
\bigl(
\Theta(y;z)\rightarrow\phi(x;y)
\bigr),$ and $
\psi_1(x;z)
:=
\exists y\,
\bigl(
\Theta(y;z)\wedge\phi(x;y)
\bigr).$
Since $\Theta$ is an $\mathcal L_\Gamma(\Gamma(M))$-formula and
$\phi$ is a $\mathcal{L}(M)$-formula, both $\psi_0$ and $\psi_1$ are
$\mathcal L(M)$-formulas. \\
Since the tuple $b$ satisfies $\Theta(y;d)$. Then, for every $a\in X(\U')$, if $\U' \models \psi_0(a;d) $ then $\U' \models \phi(a;b),$
because the universal implication defining $\psi_0(a;d)$ may be
evaluated at $y=b$. Similarly, if 
$\U' \models \phi(a;b)$ then $\U' \models \psi_1(a;d),$
because $y=b$ witnesses the existential quantifier defining
$\psi_1(a;d)$. Thus
$
\U'\models
\forall x\,
\bigl(
\psi_0(x;d)\rightarrow
\phi(x;b)\rightarrow
\psi_1(x;d)
\bigr),
$ which proves the first part of the statement.\\
It is only left to compare the traces on $M$. Let $a\in X(M)$. Suppose first that
$\U\models\phi(a;b).$
By  \textup{(1)}, $\Theta(y;d)\vdash\phi(a;y).$
Therefore
$\U'\models\psi_0(a;d).$
Together with the fact that 
$\mathbb{M} \models \forall x \big( \psi_0(x;d)\rightarrow\phi(x;b)\big)$, this shows that
$\psi_0(M;d)=\phi(M;b).$
Now suppose that $\U\models\neg\phi(a;b).$
By \textup{(2)}, $\Theta(y;d)\vdash\neg\phi(a;y).$
Hence there is no $y$ satisfying both $\Theta(y;d)$ and
$\phi(a;y)$, so $\U'\models\neg\psi_1(a;d).$
Since $\mathbb{M} \models \forall x \big( \phi(x;b)\rightarrow\psi_1(x;d)\big)$, we can conclude that
$\phi(M;b)=\psi_1(M;d).$ We conclude that
$\psi_0(M;d)=\phi(M;b)=\psi_1(M;d),$
which is the second part of the statement.
\end{proof}
\begin{proposition}
\label{prop:gs-implies-gamma-orthogonal} Let $T$, $M$ and $\Gamma$ be as in \cref{fact:strong-definitions-gamma}. Let $\mu$ be a global $M$-invariant Keisler measure. If $\mu$ is
generically stable over $M$, then $\mu$ is orthogonal to $\Gamma$.
\end{proposition}

\begin{proof}
Let $b\in\Gamma(\U)^n.$ We prove that $\mu\rest_M$ has a unique extension to $Mb$.
Fix an $\mathcal L(M)$-formula $\phi(x;y)$, where $y$ ranges over
$\Gamma^n$, and let $\epsilon>0$. By
Fact~\ref{fact:strong-definitions-gamma}, there are an elementary
extension of pairs
$
(\U,M)\preceq(\U',M'),
$
formulas $\psi_0(x;z)$ and $\psi_1(x;z)$ over $M$, and a tuple
$d\in M'$ such that
\[\U'\models
\forall x,
\bigl(
\psi_0(x;d)\rightarrow
\phi(x;b)\rightarrow
\psi_1(x;d)
\bigr),  \ \text{and}
\tag{*}\]

\[\psi_0(M;d)
= \phi(M;b)
=\psi_1(M;d).
\tag{**}\]

Since $\mu$ is generically stable over $M$, it is definable over $M$
and finitely satisfiable in $M$ (\emph{c.f} \cref{thm:genstable}).  We write $\mu$ to also indicate the canonical extension (i.e. the definable extension which is unique by \cite[Theorem~A.24]{CoGaHa}) of $\mu$ to $\mathbb{M}'$, which is still generically stable over $M$. 
We first show that the two strong definitions in \textup{(*)} have
the same $\mu$-measure as $\phi(x;b)$.  Put $r:=\mu(\phi(x;b)).$  \\

\begin{claim} \emph{$\mu(\psi_{0}(x,d))=\mu(\psi_{1}(x,d))=\mu(\phi(x,b))=r$}.\end{claim}
\begin{claimproof}
Consider $D_0(x):=\phi(x;b)\wedge\neg\psi_0(x;d).$
By \textup{(**)}, the formula $D_0(x)$ has no realization in $M$.
If
$\mu(D_0(x))>0,$
finite satisfiability of $\mu$ in $M$ would give some $a\in M^x$
satisfying $D_0(a)$, a contradiction. Hence $\mu\bigl(
\phi(x;b)\setminus\psi_0(x;d)
\bigr)
=0.$ Similarly, the formula
$D_1(x):=
\psi_1(x;d)\wedge\neg\phi(x;b)$
has no realization in $M$, again by \textup{(**)}. Therefore
$\mu\bigl(
\psi_1(x;d)\setminus\phi(x;b)
\bigr)=0.$
Combining this with \textup{(*)} we conclude that $
\mu(\psi_0(x;d))
= \mu(\phi(x;b))
= \mu(\psi_1(x;d))=r.$ \end{claimproof}

It remains to replace the parameter $d\in M'$ by parameters from
$M$. For $i=0,1$, consider the function
$F_i:S_z(M)\mapsto [0,1],$ defined as $
F_i(\tp(c/M)):=\mu(\psi_i(x;c)).$
Since $\mu$ is definable over $M$, the function $F_i$ is continuous.
By Claim~1,
$F_i(\tp(d/M))=r.$
Therefore the inverse image under $F_i$ of the interval
$(r-\epsilon/2,r+\epsilon/2)$
is an open neighborhood of $\tp(d/M)$ in $S_z(M)$. Since the Stone
topology has a basis of clopen sets, there is an
$\mathcal L(M)$-formula $\theta_i(z)\in\tp(d/M)$ such that
\[
\left|
\mu(\psi_i(x;c))-r
\right|
<
\frac{\epsilon}{2} \ \text{whenever}\ \U'\models\theta_i(c).\]

In $(\U',M')$, the tuple $d\in M'$ satisfies
$\theta_0(d)\wedge\forall x\bigl(\psi_0(x;d)\rightarrow\phi(x;b)\bigr).$
Writing $P$ for the predicate naming the distinguished submodel, this
says that
$(\U',M')\models\exists z\big(P(z)\wedge\theta_0(z)\wedge\forall x\bigl(\psi_0(x;z)\rightarrow\phi(x;b)\big).$
The parameters occurring in this formula belong to the smaller pair
$(\U,M)$. Since
$
(\U,M)\preceq(\U',M'),
$
there is $d_0\in M$ such that
$\U\models
\forall x
\bigl(
\psi_0(x;d_0)\rightarrow\phi(x;b)
\bigr)$
and, $\left|
\mu(\psi_0(x;d_0))-r
\right|
<
\frac{\epsilon}{2}.$
Applying the same argument to
$
\theta_1(z)
\wedge
\forall x
\bigl(
\phi(x;b)\rightarrow\psi_1(x;z)
\bigr),
$
we obtain $d_1\in M$ such that
$\U\models
\forall x
\bigl(
\phi(x;b)\rightarrow\psi_1(x;d_1)
\bigr)$
and
$\left|
\mu(\psi_1(x;d_1))-r
\right|
<
\frac{\epsilon}{2}.$

Consequently, 
$\U \models \forall x \big(\psi_0(x;d_0)
\rightarrow
\phi(x;b)
\rightarrow
\psi_1(x;d_1)\big).$
In particular,
$\mu(\psi_0(x;d_0))
\leq r
\leq
\mu(\psi_1(x;d_1)).$
Hence
\[
\mu(\psi_1(x;d_1))
-\mu(\psi_0(x;d_0))
=
\bigl(
\mu(\psi_1(x;d_1))-r
\bigr)
+
\bigl(
r-\mu(\psi_0(x;d_0))
\bigr)\
<
\frac{\epsilon}{2}
+
\frac{\epsilon}{2}\
=
\epsilon.
\]

Thus, for the arbitrary formula $\phi(x;b)$ and every
$\epsilon>0$, we have found $M$-definable lower and upper
approximations whose measures differ by less than $\epsilon$. By
Lemma~\ref{lem:unique-extension-criterion}, the measure
$\mu\rest_M$ has a unique extension to $Mb$. Since the tuple $b\in\Gamma(\U)^n$ was arbitrary, $\mu$ is orthogonal
to $\Gamma$.
\end{proof}

\subsection{From orthogonality to stable approximation}
\label{subsec:orthogonality-stable-approximation}

Throughout this subsection, we work with $T=\mathrm{ACVF}$. Furthermore we will assume that $M\models\ACVF$ is spherically complete,
so that
Fact~\ref{fact:implication-from-residue-value-restrictions}
applies. Let $\mu$ be a global $M$-invariant Keisler measure in a
tuple $x$ of valued-field variables.
\begin{lemma}
\label{lem:weakly-random-invariant-extension} Let $M \models \mathrm{ACVF}$ and $\mu$ be a global Kiesler $M$-invariant measure.
Let $b$ be a finite tuple from $\U$, and let
$p\in S_x(Mb)$ assume that
$\mu(\theta(x))>0$
for every $\mathcal{L}(Mb)$-formula $\theta(x)\in p$. Then $p$ does not fork over $M$.
Consequently, $p$ admits a global $M$-invariant extension.
\end{lemma}
\begin{proof}
Since $\mathrm{ACVF}$ is dependent and $\mu$ is $M$-invariant, the measure $\mu$
does not fork over $M$ (\emph{c.f.} \cite[Proposition 4.3]{nipmeasures}). In particular, every formula of positive
$\mu$-measure does not fork over $M$. Let $\theta(x)\in p$. By assumption,
$\mu(\theta(x))>0.$
Therefore $\theta(x)$ does not fork over $M$. As this holds for every
formula in $p$, the type $p$ does not fork over $M$. Let $q$ be a global non-forking extension of $p$. In a dependent theory, a
global type which does not fork over a model is invariant over that
model (\emph{c.f.}\cite[Proposition $2.1(ii)$]{genmeasures}). Hence $q$ is a global $M$-invariant extension of $p$.
\end{proof}

\begin{definition}
\label{def:stable-value-neighborhood}  Let $M$ be a spherically complete model of $\mathrm{ACVF}$. 
Let $q$ be a global $M$-invariant type in the valued-field variables
$x$, and let $b$ be a finite tuple from the valued-field sort.

A formula $\vartheta(x)\in q\rest_{Mb}$ is called a
\emph{stable--value group neighborhood of $q$ over $Mb$} if it can be
written as
$\vartheta(x)
=
\psi(x)
\wedge
\alpha\bigl(f(x);d\bigr)
\wedge
\beta\bigl(g(x);e\bigr),$
where:

\begin{enumerate}
\item $\psi(x)$ is an $\mathcal L(M)$-formula belonging to
$q\rest_M$;
\item $D$ is a finite product of sorts of $\St_M$, and
$f:\psi(\U)\mapsto D$
is an $M$-definable map;
\item $d\in\St_M(Mb)$
and $\alpha(u;d)$ is a formula in the multi-sorted stable structure
$\St_M$;
\item for some $r<\omega$,
$g:\psi(\U)\mapsto\Gamma^r$
is an $M$-definable map;
\item $e\in\Gamma(Mb)$
and $\beta(v;e)$ is a formula in the pure language of ordered
divisible abelian groups.
\end{enumerate}
\end{definition}
\begin{lemma}
\label{lem:existence-stable-value-neighborhood} Let $M$ be a spherically complete model of $\mathrm{ACVF}$. 
Let $q$ be a global $M$-invariant type in the valued-field variables,
let $b$ be a finite tuple from the valued-field sort, and let
$\phi(x;b)\in q.$
Then there is a stable--value group neighborhood $\vartheta_q(x)$ of $q$
over $Mb$ such that $\vartheta_q(x)\rightarrow\phi(x;b).$
\end{lemma}

\begin{proof}
By \cref{fact:implication-from-residue-value-restrictions}, $q\rest_{M} \cup q^{st} \rest_{Mb} \cup q^{\Gamma}\rest_{Mb} \vdash q \rest_{Mb}$. 
Since $\phi(x;b)\in q\rest_{Mb}$, by compactness there are  finitely many
formulas from these three partial types already imply
$\phi(x;b)$. Let $\psi(x)$ be the conjunction of the finitely many formulas chosen
from $q\rest_M$. Only finitely many coordinates of the stable and
value-group pro-definable maps occur in the remaining formulas (\emph{c.f.} \cref{not:residue-value-prodefinable-maps}).
After including their domains in $\psi(x)$, these coordinates give
$M$-definable maps $f:\psi(\U)\mapsto D$ and $g:\psi(\U)\mapsto\Gamma^r,$ where $D$ is a finite product of sorts of $\St_M$.
Since $\St_{M}$ is stably embedded and eliminates imaginaries (\cref{fact: stablepartproperties}), the finite conjunction of
the chosen stable formulas in $q^{st}\rest_{Mb}$ may be written as
$\alpha(f(x);d)$ for some $d\in\St_M(Mb).$
Similarly, stable embeddedness and elimination of imaginaries in the value group (\emph{c.f.} \cref{fact:acvf-stably-embedded-sorts}), allow the
finite conjunction of the chosen value-group formulas to be written
as
$\beta(g(x);e)$ for some $e\in\Gamma(Mb),$
where $\beta$ is a formula in the pure ordered-group language.

Put
$\vartheta_q(x)
:=
\psi(x)
\wedge
\alpha(f(x);d)
\wedge
\beta(g(x);e).$
Each conjunct belongs to $q$, so
$\vartheta_q(x)\in q.$
By the choice of the finite subsets, one also has
$\vartheta_q(x)\rightarrow\phi(x;b).$
Thus $\vartheta_q$ is the required stable--value group neighborhood.
\end{proof}

Whenever the measure $\mu$ is orthogonal to the value group, then 
we can replace the
value-group part of such a neighborhood by formulas over $M$, while
leaving its stable part unchanged.
\begin{lemma}
\label{lem:stable-value-neighborhood-approximated}
Assume that $\mu$ is orthogonal to $\Gamma$. Following the notation introduced in \cref{def:stable-value-neighborhood}. Let
$\vartheta_{q}(x)
=
\psi_{q}(x)
\wedge
\alpha_{q}(f(x);d)
\wedge
\beta_{q}(g(x);e)$
be a stable--value group neighborhood of a global $M$-invariant type $q$. Then
$\vartheta_{q}(x)$ is $\DEFST_{M}(\mathbb{M})$-approximated with respect to $\mu$. More precisely, for every $\epsilon>0$, there are
$\vartheta^{q}_-(x),\vartheta^{q}_+(x)\in \DEFST_{M}(\mathbb{M})$ 
such that
$\vartheta^{q}_-(x)
\rightarrow
\vartheta_{q}(x)
\rightarrow
\vartheta^{q}_+(x)$
and
$\mu(\vartheta^{q}_+(x))-\mu(\vartheta^{q}_-(x))<\epsilon.$
\end{lemma}

\begin{proof}
Let
$\rho_{q}(x;e):=\psi_{q}(x)\wedge\beta_{q}(g(x);e).$
The only parameter of $\rho(x;e)$ outside $M$ is the finite tuple $e\in\Gamma(Mb).$
Since $\mu$ is orthogonal to $\Gamma$, the measure
$\mu\rest_M$ has a unique extension to $Me$. By
Lemma~\ref{lem:unique-extension-criterion} (applied to $\mathbb{F}=\mathbb{B}_{x}(M)$), for every
$\epsilon>0$ there are $\mathcal L(M)$-formulas
$\rho^{q}_-(x)$ and $\rho^{q}_+(x)$ such that
$\rho^{q}_-(x)
\rightarrow
\rho_{q}(x;e)
\rightarrow
\rho^{q}_+(x)$ and
$\mu(\rho^{q}_+(x))-\mu(\rho^{q}_-(x))<\epsilon.$
Define
$\vartheta^{q}_-(x)
:=
\rho^{q}_-(x)\wedge\alpha_{q}(f(x);d)$
and
$\vartheta^{q}_+(x)
:=
\rho^{q}_+(x)\wedge\alpha_{q}(f(x);d).$
The formula $\alpha_{q}(f(x);d)$ is stable and $d \in \St_{M}(\mathbb{M})$,
and $\rho^{q}_-$ and $\rho^{q}_+$ are formulas over $M$. Hence
$\vartheta^{q}_-(x),\vartheta^{q}_+(x)\in \DEFST_{M}(\mathbb{M}).$
The implication $\vartheta^{q}_-(x) \rightarrow \vartheta_{q}(x)\rightarrow \vartheta^{q}_+(x)$ is immediate.  Moreover,
$\vartheta^{q}_+(x)\setminus\vartheta^{q}_-(x)
\subseteq
\rho^{q}_+(x)\setminus\rho^{q}_-(x).$
Therefore:
\[
\mu(\vartheta^{q}_+(x))-\mu(\vartheta^{q}_-(x))=
\mu\big(
\vartheta^{q}_+(x)\setminus\vartheta^{q}_-(x)
\big)\leq
\mu\big(
\rho^{q}_+(x)\setminus\rho^{q}_-(x)
\big)=
\mu(\rho^{q}_+(x))-\mu(\rho^{q}_-(x))<
\epsilon.
\]
\end{proof}
Given a formula $\phi(x,b)$ and $q$ a global $M$-invariant type \cref{lem:existence-stable-value-neighborhood} gives us a stable-value group neighborhood $\vartheta_{q}(x)$ of $q$ that implies $\phi(x,b)$. \cref{lem:stable-value-neighborhood-approximated} show us how to approximate $\vartheta_{q}(x)$ by formulas in $\DEFST_{M}(\mathbb{M})$.  
We now pass from the local stable--value group neighborhoods to give a complete approximation of $\phi(x,b)$. 

\begin{proposition}
\label{prop:stable-inner-approximation} Let $M$ be a spherically complete model of $\mathrm{ACVF}$ and let  $\mu$ be a global $M$-invariant measure that is orthogonal to $\Gamma$. Let $\phi(x;b)$ be an
arbitrary formula and let $\epsilon>0$. Then there is a formula
$\sigma(x)\in\mathbb \DEFST_{M}(\U)$
such that
$\sigma(x)\rightarrow\phi(x;b)$
and $\mu(\phi(x;b))-\mu(\sigma(x))<\epsilon.$
\end{proposition}

\begin{proof}
Let
$\mu_b:=\mu\rest_{Mb}$
and let
$
K_{\phi,b}
:=
\left\{
p\in\operatorname{Supp}(\mu_b):
\phi(x;b)\in p
\right\}.$
The support $\operatorname{Supp}(\mu_b)$ is closed in
$S_x(Mb)$, and hence $K_{\phi,b}$ is compact (\emph{c.f.} \cref{fact:support-keisler-measure}). Fix $p\in K_{\phi,b}.$
By \cref{lem:weakly-random-invariant-extension}, choose a global
$M$-invariant extension $q_p$ of $p$. Since
$\phi(x;b)\in q_p,$ by
\cref{lem:existence-stable-value-neighborhood} we obtain a stable-value neighborhood
$\vartheta_p(x)\in p$
such that
$\vartheta_p(x)\rightarrow\phi(x;b).$ The clopen sets
$\{[\vartheta_p] \ | \ p\in K_{\phi,b}\}$
cover $K_{\phi,b}$. By compactness, choose
$p_1,\ldots,p_n\in K_{\phi,b}$ such that
$K_{\phi,b}\subseteq \bigcup_{i=1}^n[\vartheta_{p_i}].$
Let $\eta(x):=\bigvee_{i=1}^n\vartheta_{p_i}(x).$
Then $\eta(x)\rightarrow\phi(x;b).$\\

\begin{claim}
\emph{$\mu(\phi(x;b)\wedge\neg\eta(x))=0.$}    
\end{claim} 
\begin{claimproof}
Suppose otherwise. Since
$\phi(x;b)\wedge\neg\eta(x)$ is a formula over $Mb$, we have
$\mu_b\bigl(\phi(x;b)\wedge\neg\eta(x)\bigr)>0.$
By Fact~\ref{fact:support-keisler-measure}, there is a type
$p\in\operatorname{Supp}(\mu_b)$
such that $\phi(x;b)\wedge\neg\eta(x)\in p.$
Since $\phi(x;b)\in p$, we have $p\in K_{\phi,b}.$
The finite cover constructed above therefore implies
$\eta(x)\in p$. On the other hand,
$\neg\eta(x)\in p$, a contradiction. This concludes the proof of the Claim.\end{claimproof}

For each $i\leq n$, apply
Lemma~\ref{lem:stable-value-neighborhood-approximated} to
$\vartheta_{p_i}(x)$ and  $\frac{\epsilon}{n}$, to obtain
$\vartheta_i^-(x),\vartheta_i^+(x) \in\DEFST_{M}(\U)$
such that
$\vartheta_i^-(x)\rightarrow\vartheta_{p_i}(x)\rightarrow\vartheta_i^+(x)$
and
$\mu(\vartheta_i^+(x))-\mu(\vartheta_i^-(x))<\frac{\epsilon}{n}.$
Put $\sigma(x):=\bigvee_{i=1}^n\vartheta_i^-(x).$
Then $\sigma\in \DEFST_{M}(\U)$ and, $\U \models \forall x \big(\sigma(x)\rightarrow \eta(x) \rightarrow \phi(x;b)\big).$ Moreover, $\eta(\U)\setminus\sigma(\U)
\subseteq
\bigcup_{i=1}^n
\big(
\vartheta_{p_i}(\U)\setminus\vartheta_i^-
(\U)\big).$ For each $i \leq n$, we have
$
\vartheta_{p_i}(\U)\setminus\vartheta_i^-
(\U)\subseteq
\vartheta_i^+(\U)\setminus\vartheta_i^-(\U),
$
and hence
$
\mu\big(
\vartheta_{p_i}(x)\setminus\vartheta_i^-(x)
\big)
<
\frac{\epsilon}{n}.
$
By the claim and finite additivity, we conclude that
\[
\mu(\phi(x;b))-\mu(\sigma(x))
=
\mu(\eta(x))-\mu(\sigma(x))=
\mu(\eta(x)\wedge \neg \sigma(x))\leq
\sum_{i=1}^n
\mu\big(
\vartheta_{p_i}(x)\setminus\vartheta_i^-(x)
\big)<
\epsilon.\]
\end{proof}

We can now obtain both an inner and an outer $\DEFST_{M}(\U)$-approximations. 

\begin{proposition}
\label{prop:orthogonal-then-stably-dominated} Assume that $M$ is a spherically complete model of $\mathrm{ACVF}$. Let $\mu$ be a global $M$-invariant Keisler measure. If $\mu$ is
orthogonal to $\Gamma$, then $\mu$ is stably approximated over $M$.
Consequently, $\mu$ is stably dominated over $M$.
\end{proposition}

\begin{proof}
Let $\phi(x;b)$ be a formula and let $\epsilon>0$. By \cref{prop:stable-inner-approximation} applied to $\phi(x;b)$ and $\epsilon/2$, we obtain
$\sigma_-(x)\in \DEFST_{M}(\U)$
such that
$\sigma_-(x)\rightarrow\phi(x;b)$
and $\mu(\phi(x;b))-\mu(\sigma_-(x))<\frac{\epsilon}{2}.$ Applying the same proposition to $\neg\phi(x;b)$ and $\epsilon/2$, we obtain
$\tau(x)\in \DEFST_{M}(\U)$ such that
$\tau(x)\rightarrow\neg\phi(x;b)$
and
$\mu(\neg\phi(x;b))-\mu(\tau(x))<\frac{\epsilon}{2}.$
Set $\sigma_+(x):=\neg\tau(x).$
Since $\DEFST_{M}(\U)$ is a Boolean algebra,
$\sigma_+(x)\in \DEFST_{M}(\U)$. 
Moreover,
$\phi(x;b)\rightarrow\sigma_+(x),$
then  $\mu(\sigma_+(x))-\mu(\phi(x;b))<\frac{\epsilon}{2}.$
Combining the previous results, we get
$\sigma_-(x)\rightarrow\phi(x;b)\rightarrow\sigma_+(x)$
and
\[
\mu(\sigma_+(x))-\mu(\sigma_-(x))
=
\bigl(
\mu(\sigma_+(x))-\mu(\phi(x;b))
\bigr) +
\bigl(
\mu(\phi(x;b))-\mu(\sigma_-(x))
\bigr)
<\epsilon.\]
Thus $\mu$ is $\DEFST_{M}(\U)$-approximated, and hence is stably
approximated over $M$.

The equivalence between stable approximation and stable domination from \cref{lem:unique-extension-criterion}, now shows
that $\mu$ is stably dominated over $M$.
\end{proof}

\begin{proof}[Proof of
Theorem~\ref{thm:main-generically-stable-measures-acvf}]
Any stably dominated measure is generically stable by \cref{prop:stableapproxthengenstable}, this is the implication $(1) \rightarrow (2)$. The implication $(2)\rightarrow (3)$, i.e from generic stability to orthogonality to the value
group is Proposition~\ref{prop:gs-implies-gamma-orthogonal}. Finally, the implication $(3)\rightarrow (1)$, this is from orthogonality to the value group to stable approximation is
\cref{prop:orthogonal-then-stably-dominated}. Lastly, stable
approximation is equivalent to stable domination by
\cref{lem:unique-extension-criterion}. 
\end{proof}

\section{Generically stable measures over \texorpdfstring{$\R_\infty$}{R} and Borel measures}
In the next chapter, we will work with generically stable measures over algebraically closed valued fields and prove that when the 
model is spherically complete and the valued group is $(\R +, \leq, 0)$, these correspond to Radon measures over the Hrushovski-Loeser Berkovich spaces. The proof will rely strongly on the relation between smooth Keisler measures over the o-minimal structure $\mathbb{R}_{\infty}$ and   Borel measures over $\R_{\infty}$. Such characterization builds on the previous identification obtained for o-minimal expansions of the real closed field (\cite[Theorem $6.3$]{genmeasures}, \cite[Proposition $5.5$]{SimonGenericallyStable}). \\
Consider any o-minimal expansion of the real numbers. By Cell Decomposition, all definable subsets of $\R^n$ are finite unions of cells $C$, each of which is homeomorphic to an open subset of $\R^k$ for $k<n$. In particular all the sets $C(\R)$ (and therefore the $\R$-points in any definable set) are Borel subsets of $\R^n$.\\
So if we have a $\sigma$-additive measure $\mu$ on the Borel subsets of $\R$, we can define a Keisler measure $\lambda$ on the boolean algebra of $\R$-definable subsets by  $\lambda(D):=\mu(D(\R))$, which by definition will be a countably additive Keisler measure over $\R$ (see \cite[Definition $6.1$]{genmeasures}). This measure $\lambda$ is a smooth Keisler measure by the following fact.
\begin{fact}\label{borelandcountablyadditivethensmooth}\cite[Theorem 6.3]{genmeasures}
Let $M$ be an o-minimal expansion of the real field. Let $V$ be a definable set in $M$, and $\lambda$ be a countably additive Keisler measure on $V$ over $\R$. Then $\lambda$ is smooth.
\end{fact}
Hence, so far, we have shown that any Borel measure $\mu$ over $\R$ induces a countably additive Keisler measure $\lambda$ over $\R$, and such measures are always smooth.\\
For the converse, let $\lambda$ be Keisler measure on $\mathbb B_x(\R)$. By Keisler's results we may assume that $\lambda$ is $\sigma$-additive on the $\sigma$-algebra $\sigma B_x(\R)$ generated  by $\R$-definable subsets of $\U^{|x|}$. One could try to define a Borel measure taking the $\R$-points of sets in $\sigma B_x(\R)$, but it may happen that two sets with different $\lambda$-measures have the same real points. For example, consider a Keisler probability measure $\lambda$ concentrating on the set of positive infinitesimals $0^+$. Then $1=\lambda((0,1))\neq \lambda(\bigcup_{i\in \N} (1/n,1))=0$ but the $\R$ points of both sets are equal.\\
If $D$ is bounded definable set, then it is contained in a closed box $B$. The standard map $\mathrm{st} \colon B(\mathbb{M}) \to B(\mathbb{R})$ is well defined, and so it is its restriction from $D(\mathbb{M})$ to $st(D(\mathbb{M})) \subseteq B(\mathbb{R})$. Taking coordinates, the pre-image of a closed $\R$-interval is a countable intersection of definable sets, thus the standard map is a measurable map from the $\sigma$-algebra generated by the definable subsets of $B$ to the Borel $\sigma$-algebra of $B(\mathbb{R})$. Thus, if we fix a bounded closed box $B$ we can define a measure $\mu$ on the Borel subsets of $B(\mathbb{R})$ by $\mu(X(\R)):=\lambda \big(st^{-1}(X(\R))\big)$. Note that the Borel measure $\mu$ always induces a Kiesler measure on $\mathbb{B}(\mathbb{R})$ the boolean algebra of $\mathbb{R}$-definable subsets of $B$, by taking its restriction, and we denote it to be $\mu_{*}$. If $\lambda$ is smooth, then the induced Borel measure $\mu$ agrees on the $\mathbb{R}$-definable subsets of $B(\mathbb{R})$, i.e. $\mu_{*}$ gives back $\lambda$,\footnote{In the above example the measure is \emph{not} smooth. Since $st^{-1}((0,1))$ doesn't include any infinitesimal points around zero we would have  
$\mu(0,1)^\R:=\lambda(st^{-1}((0,1)))=\lambda\left(st^{-1}\left(\bigcup_{i\in \N} \left(1/n,1\right)\right)\right)=0,$ so that if $\lambda_\mu$ is the Keisler  measure induced by $\mu$ we have $\lambda_\mu((0,1))=\mu((0,1)^\R)=0\neq 1=\lambda((0,1)).$ } as was proved in the following fact:

\begin{fact}\label{smooththenborel}\cite[Proposition 5.5]{SimonGenericallyStable}
Let $M$ be an o-minimal expansion of the real field. Let $\mu$ be a smooth measure over $M^k$ concentrating in $[0,1]^k$. Then there is $\mu_{0}$ a $\sigma$-additive $\R$-Borel measure on $[0,1]^k$ such that $\mu_0$ induces $\mu$ as a Keisler measure.
\end{fact}

We will need to obtain analogues of  \cref{smooththenborel} and \cref{borelandcountablyadditivethensmooth} for the structure $(\mathbb{R}_{\infty}, +, <)$: in order to have the valuation $\nu$ be defined everywhere, we need to work with $\R_\infty$, and in a valued field the structure on $\Gamma$ is only that of an additive group.\\
We will always assume that we are working in what in \cite{tamegeo} are defined as \emph{bounded subsets of $\Ri$}; namely, definable sets contained inside ``boxes'', namely, subsets of $(\R_\infty)^k$ of the form $\prod_{i=1}^k I_i$ where $I_i$ is a closed interval of $\R_\infty$, so that $I_{i}=[a_i,b_i]$ where $a_i, b_i \in \mathbb{R} \cup \{\infty\}$. 
We will use Keisler's approach and work with points in a saturated extension $\U$ of $(\Ri, +, \leq, 0)$. Notice that if $B$ is a box then the standard part map is a well defined map from $B(\M)$ to $B(\R)$, where we naturally expand the usual standard map $\mathrm{st}(a)=\infty$ whenever $a>n$ for all $n\in \mathbb N$.\\ 
To achieve a generalization of \cref{smooththenborel} we will need to define a family of neighborhoods that makes sense in the context of adding $\infty$ to the structure and which are definable using only the additive structure we assume.  We define the local neighborhoods as follows. Let $n\in \N$ and let $a \in \Gamma(\U)$ be such that $a \neq \infty$: we will define its closed neighborhood $U_{1/n}(a):=\{x\mid a-1/n\leq x\leq a+1/n\}$ and its open neighborhood $U_{<1/n}(a):=\{x\mid a-1/n< x< a+1/n\}$. While, for the point $\infty$ we take  $U_{1/n}(\infty):=\{x\mid x\geq n\}$, $U_{<1/n}(\infty):=\{x\mid x> n\}$.\\
We extend the definition by cartesian powers to tuples in $\Gamma_{\infty}(\U)^k$, and for a definable $X$ we define the $1/n$-neighborhood of $X$ to be the definable set
\[U_{<1/n}(X):=\{x\mid \exists y\in X\ x\in U_{<1/n}(y)\}, \ \text{and similarly for $U_{1/n}$}.\]
Given $F$ a closed $\R_{\infty}$-definable subset of $\R_{\infty}^{r}$ then $U_{<1/n}(F)$ is an open $\R_{\infty}$-definable neighbourhood of $F$, and the sequence $U_{< 1/n}(\mathbb{R}_{\infty})$ is decreasing.\\
Moreover, $F(\mathbb{R}_{\infty})=\bigcap_{n \geq 1} U_{<1/n}(F(\mathbb{R}_{\infty}))$.

\begin{theorem}\label{theo: radon and smooth in R}
Let $M$ be an o-minimal expansion of $(\R_\infty, +, \leq,0)$. 
\begin{enumerate}
\item If $V$ is an $\R_{\infty}$-definable set contained in a (closed) box, and $\mu$ a countably additive Keisler measure on $V$ over $\R_{\infty}$, then $\mu$ is smooth. In particular, the Keisler measure on the $\R$-definable subsets of a box $B$ induced by a Borel measure on $B$ is smooth.
\item Conversely, if $\mu$ be a smooth measure over $(\R_\infty)^k$ concentrating in a (closed) box $B$, there is a $\sigma$-additive $\R_{\infty}$-Borel measure on $B(\R_\infty)$ such that $\mu_0$ induces $\mu$ as a Keisler measure.
\end{enumerate}
\end{theorem}

\begin{proof}
    The proofs are inspired and quite close to the original ones in \cite[Theorem $6.3$]{genmeasures} and \cite[Proposition 5.5]{SimonGenericallyStable}, so we will sketch the common parts and concentrate on the things we needed to add. As usual, given a subset $X$ in $\mathbb{R}^{r}_{\infty}$ we write $\overline{X}$, $X^{\mathrm{o}}$ to denote their topological closure and its open interior. \\
    
We first prove $(1)$. Let $\mu$ be a countably additive Keisler measure over $\R_{\infty}$-definable sets concentrating on an $\R_{\infty}$-definable (closed) box $B$. 
 Let $\mu'$ be any extension of $\mu$ to a global Keisler measure, then for any $\mathbb{R}_{\infty}$-definable set $X$ we have $\mu'(X)=\mu(X)$. We need to prove that for any $\M$-definable subset $D$, $\mu'(D)$ is uniquely determined by $\mu$. Let $D$ be any $\M$-definable subset inside $B$. The proof will be an induction on the o-minimal dimension of $V$. The standard map is well defined $\mathrm{st}: B(\U) \to \B(\R_{\infty})$, hence we can consider $\mathrm{st}(D(\U)) \subseteq B(\R_{\infty})$. Then $st(D(\M))$ is a subset of $B(\R_{\infty})$ which, as explained in \cite[Theorem $6.3$]{genmeasures} (see the paragraph above the first claim) is $\R$-definable by definability of types in $\mathbb{R}$ (which was proved in \cite{marker-steinhorn}). In particular, $st(D(\M))$ and $st(D^c(\M))$, and their intersection $st(D(\U)) \cap st(D^{c}(\U))$ are $\R$-definable closed subsets of $B$. As in the second part of the paragraph before the first claim in \cite[Theorem $6.3$]{genmeasures}, if we let $Y, D_1$ and $D_0$ be the $\R$-definable sets such that 
$Y(\mathbb{R}_{\infty})=st(D(\U)) \cap st(D^{c}(\U)),$ $D_{1}(\mathbb{R}_{\infty})=st(D(\U)) \setminus Y$, and $D_{0}(\mathbb{R}_{\infty})=st(D^{c}(\U)) \setminus Y$, then $V$ is the disjoint union of $D_{0}$, $D_{1}$ and $Y$. \\ 

If $st(a) \in \overline{st(D(\U))}\setminus (st(D(\U)))^\circ$ then $U_{1/n}(a)$ intersects $D^c$ so in particular $a\in st(D^c)$ and both  $D_1(\R_{\infty}):=st(D(\U))\setminus st(D^c(\U))$ and $D_0(\R_{\infty}):=st(D^c(\U))\setminus st(D(\U))$ are open. 

The proof of Claim 1 in \cite[Theorem $6.3$]{genmeasures} holds and so that the set  $D_{2}:=Y \cup \left(\overline{{D_0}}\cap \overline{{D_1}}\right)$ has dimension less than $n$ because neither $Y(\Ri)$ nor $\overline{{D_0}(\Ri)}\cap \overline{{D_1}(\Ri)}$ can contain 
an $\R$-definable open set. So by the induction hypothesis $\mu'$ is determined by $\mu$ on definable subsets of $D_2(\U)$.
The proof of Claim 3 in \cite[Theorem $6.3$]{genmeasures} also goes through, and proves that  for each $n \in \mathbb{N}$ we have $D(\U)\setminus U_{<1/n}(D_2(\U))=D_1(\U)\setminus U_{<1/n}(D_2(\U))$. Both $D_{1}(\U)$ and $U_{< 1/n}(D_{2}(\U))$ are $\mathbb{R}_{\infty}$-definable so that (identifying as Keisler does a measure on a definable set with the points satisfying said set in a saturated model to achieve a $\sigma$-additive extension) $\mu'(D_{1}(\U)\setminus  U_{<1/n}(D_2(\U))=\mu(D_{1}(\M)\setminus D_{2}(\U))$.  \\
Now, $D_2(\R_\infty)$ is closed, $\R$-definable and $D_2(\R_\infty)=\bigcap_{n\in \N} U_{1/n}(D_2)(\R_\infty)$. We assumed that $\mu$ was countably additive over $\R_{\infty}$ and these are all $\R_{\infty}$-definable sets so  $\mathrm{lim}_{n \rightarrow \infty} \mu(U_{1/n}(D_2)\setminus D_2)=0$. For any $n \in \mathbb{N}$
\begin{align*}
0 &\leq \mu'\Bigl(
D(\U)\cap
\bigl(U_{<1/n}(D_{2}(\U))\setminus D_2(\U)\bigr)
\Bigr)\\
&\leq \mu'\bigl(U_{1/n}(D_{2}(\U))\setminus D_2(\U)\bigr)=\mu\bigl(U_{<1/n}(D_{2}(\M))\setminus D_2(\M)
\bigr),
\end{align*}
and $\mathrm{lim}_{n \rightarrow \infty} \mu'\Bigl(
D(\U)\cap
\bigl(U_{<1/n}(D_{2}(\U))\setminus D_2(\U)\bigr)
\Bigr)=0$.\\
On the other hand, for every $n \in \mathbb{N}$, note that we have \[D(\U)\setminus D_{2}(\U)=\big( D(\U) \setminus U_{<1/n}(D_{2}(\U)) \sqcup \big( D(\U) \cap (U_{<1/n}(D_{2}(\U)) \setminus D_{2}(\U)\big).\] Consequently, we get $\mu'\big( D(\U)\setminus D_{2}(\U)\big)= \mathrm{lim}_{n \rightarrow \infty} \mu'\big( D(\U) \setminus U_{<1/n}(D_{2}(\U))\big)$.\\
Since $\mu'\big( D(\U) \setminus U_{<1/n}(D_{2}(\U))\big)=\mu\big( D_{1}(\U) \setminus U_{1<n}(D_{2}(\U))\big)$, so by the countably additive over $\Ri$ hypothesis, 
$\mu'(D(\U) \setminus D_{2}(\U))=\mathrm{lim}_{n \rightarrow \infty} \mu\big(D_{1}(\U)\setminus U_{<1/n}(D_{2}(\U))\big).$
As $D_1\setminus U_{<1/n}(D_2)$ is an increasing sequence of $\R_{\infty}$-definable sets converging to $D_1\setminus D_2$. 
Combining this with the fact that $\mu'\bigl(D(\U)\cap D_2(\U)\bigr)$ was uniquely determined by induction hypothesis, we get that $\mu'(D(\U))$ is uniquely determined by $\mu$, as required.\\

For $(2)$, assume now that $\mu$ is a smooth measure over $(\R_\infty)^k$ concentrating in a box $B$. As in \cite[Proposition $5.5$]{SimonGenericallyStable}, the standard part map is a Borel map from the $\sigma$-algebra generated by the definable subsets in $B(\M)$ to the Borel $\sigma$-algebra of $B(\R_{\infty})$. We take $\mu_0$ to be the pushforward of $\mu$, then $\mu_{0}$ is a $\sigma$-additive measure on $B(\R_{\infty})$. We need to show that $\mu_0(X)=\mu(X)$ for any $\R_{\infty}$-definable set, so that the Keisler measure induced by the Borel measure $\mu_0$ is indeed $\mu$.
Since $\overline X\setminus X$ always has smaller dimension than $X$, it is enough to prove the result for closed $\R_{\infty}$-definable sets $X$. Indeed, this allow us to obtain the result (as in  \cite[Proposition $5.5$]{SimonGenericallyStable}), by induction on the o-minimal dimension of $X$.\\

So let $X$ be a $\R_{\infty}$-definable closed set. By definition of $\mu_{0}$ we have $\mu_0(X(\R_{\infty}))=\mu(st^{-1}(X(\R_{\infty}))$ so we need to show that $st^{-1}(X(\M))\setminus X(\M)$ has $\mu$-measure zero. In \cite[Proposition $5.5$]{SimonGenericallyStable} this was done by taking the distance from $d_X(a)$ from any point $a$ to $X$. If $st^{-1}(X(\M))\setminus X(\M)$ then 
$d_X\left(st^{-1}\left(X\left(\M\right)\right)\setminus X\left(\M\right)\right)$ is contained in the type of $0+$ over $\R$, so $st^{-1}(X(\M))\setminus X(\M)$ must have $\mu$-measure zero since the pushforward $\mu_{d_X}$ of $\mu$ to $\R$ given by the definable function $d_X$ is also smooth so $\mu_{d_X}(0^+)$ must be zero. 

Our strategy is to repeat this proof taking into account $\infty$, and covering $st^{-1}\left(X\left(\M\right)\right)\setminus X\left(\M\right)$ with sets which are forced to have measure zero because they are pushforwards of unrealized types in $\mathbb{R}_{\infty}$. 

\begin{claim}
Let $\pi$ be any coordinate projection. Then $\mu\left(\pi^{-1}(st^{-1}(\infty))\setminus \pi^{-1}(\infty)\right)=0.$   
\end{claim}

\begin{claimproof}
If $a\in \pi^{-1}(st^{-1}(\infty))\setminus \pi^{-1}(\infty)$  then the type of $\pi(a)$ over $\Ri$ is the unrealized type 
$\infty^-$ given by $\{x>n\mid n\in \N\}\cup \{x< \infty\}$. Since the pushforward $\mu_{\pi}$ of $\mu$ under $\pi$ must be smooth,  $\mu_{\pi}(\infty^-)=0$ then $\mu(\pi^{-1}(\infty^-))=\mu\left(\pi^{-1}(st^{-1}(\infty))\setminus \pi^{-1}(\infty)\right)=0.$\end{claimproof}\\

Let $I_\infty$ be the set of indices $i$ for which the corresponding
interval defining $B$ has upper endpoint $\infty$. We write $\pi_{i}$ to denote the projection map to the $i$-th coordinate. For each
$i\in I_\infty$, let $E_{i}=\{ a \in st^{-1}(X)\setminus X \ | \ st(\pi_{i}(a))=\infty \wedge \pi_{i}(a) \neq \infty\}$. By the Claim, $E_{i}$ has $\mu$-measure zero for each $i \in I_{\infty}$. Consequently, $E=\bigcup_{i \in I_{\infty}}E_{i}$ has also $\mu$-measure zero. For each non-empty $J \subseteq I_{\infty}$ we define 
\[C_{J}=\{ x \in B\ | \ \pi_{j}(x)=\infty \ \text{for all $j \in J$ and } \ \pi_{j}(x)< \infty \text{ \ for all $j \in I_{\infty}\setminus J$}\}.\] 
Note that the $(C_{J}(\mathbb{R}_{\infty}) \ | \ J \subseteq I_{\infty})$ form a partition of $B(\mathbb{R}_{\infty})$. Let $X_{J}= C_{J} \cap X$.  Note that the set \(X_J\) is relatively closed within \(C_J\).  Furthermore, if $a \in st^{-1}(X)\setminus X$ then $a\not\in E$ implies that $st(\pi_i(a))=\infty$ only if $\pi_i(a)=\infty$, so that
\[st^{-1}(X) \setminus X \subseteq E \cup \bigcup_{J \subseteq I_{\infty}} \left(st^{-1}\left(X_{J}\right) \cap C_{J}\right)\setminus X_J.\]
Since $E$ has measure $0$, the following claim is enough to show that $st^{-1}(X) \setminus X$ has $\mu$-measure zero whenever $X$ is closed.

\begin{claim}
    $\mu(\big(st^{-1}(X_{J}) \cap C_{J}\big) \setminus X_{J})$ has $\mu$-measure zero.
\end{claim}

\begin{claimproof}
$X_{J}$ is relatively closed within $C_{J}$, and our argument is quite similar argument as in \cite[Proposition $5.5$]{SimonGenericallyStable}. We define, for $X_{J} \neq \emptyset$, the definable function  $d_{X_{J}}$  in the ordered additive structure with domain $C_{J}$ by \[d_{X_{J}}(z):= \mathrm{\inf}_{y \in X_{J}}\left\{ \max \left\{|\pi_{i}(z)-\pi_{i}(y)| \ | \  i \notin J \right\}\right\}= \mathrm{min}_{y \in X_{J}}\left\{ \max \left\{|\pi_{i}(z)-\pi_{i}(y)| \ | \  i \notin J \right\}\right\}.\] Since the infinite coordinates are fixed, note that if
$a\in\bigl(\operatorname{st}^{-1}(X_J)\cap C_J\bigr)
\setminus X_J$,
then $d_{X_J}(a)$ realizes the type $0^+$. Indeed, let $b=st(a)$ since $b \in X_{J}$ we have that $d_{X_{J}}(a) \leq \mathrm{max}_{i \notin J} |a_{i}-b_{i}|$. Every coordinate on the right-hand side of this equation has finite standard part, thus such minimum is an infinitesimal. Moreover, since $a \notin X_{J}$ and $X_{J}$ is relatively closed within $C_{J}$, then $d_{X_{J}}(a)>0$.\\ 
As before, the pushforward $d_{X_{J}}(\mu)$ is a smooth measure so by Fact \ref{fact:smoothmeasurezero} the realizations type $\{0^{+}\}$ must have $d_{X_{J}}(\mu)$-measure zero (\emph{c.f.} \cref{fact:smoothmeasurezero}). Consequently, \[\mu\bigl(
  (\operatorname{st}^{-1}(X_J)\cap C_J)\setminus X_J
\bigr)\leq \mu\bigl(d_{X_J}^{-1}(0^+(\U))\bigr)=0.\]
 \end{claimproof}

 This concludes the proof for the closed sets. The rest of the argument for arbitrary definable sets follows by induction on the o-minimal dimension of definable sets as in the last paragraph of \cite[Proposition $5.5$]{SimonGenericallyStable}. 
\end{proof}
\begin{remark}\label{rmk:uniquenessRadon} The Borel measure obtained in \cref{theo: radon and smooth in R}.(2) is unique. Indeed, note that the $\sigma$-algebra generated by the $\mathbb{R}_{\infty}$-definable subsets of $B(\mathbb{R}_{\infty})$ coincides with the $\sigma$-algebra of Borel sets of the topology on $B(\mathbb{R}_{\infty})$ inherited from $\mathbb{R}_{\infty}^{r}$. We can show this by double inclusion, for the left to right direction note that this follows by cell decomposition. For the converse, let  $\mathcal{P}$ be the family of sets of the form:
\[ P_{\bar{q}}\colon=\{ x \in B(\mathbb{R}_{\infty}) \ | \ x_{i} \leq q_{i} \ \text{for  $i \leq k$}\} \qquad \bar{q} \in (\mathbb{Q}\cup \{ \infty\})^{k},\]
where the condition $x_{i} \leq \infty$ imposes no restriction. These sets are definable using parameters in $\mathbb{R}_{\infty}$ and the order, and are closed under finite intersection. Furthermore, they generate the $\sigma$-algebra of Borel subsets of $B(\mathbb{R}_{\infty})$. For the converse, the rational cuts generate the topology on each coordinate, including neighborhoods of infinity as $(q,\infty]=\mathbb{R}_{\infty} \setminus \{ x \ | \ x \leq q\}$. And the finite product $\mathbb{R}_{\infty}^{r}$ has a countable basis built from taking finite products of these intervals: $(q_{1} ,q_{2})$ or $(q,\infty]$ where $q_{1},q_{2},q \in \mathbb{Q}$. Thus, applying \cref{fact:uniquenessmeasures}, to $\mathcal{P}$ and $\mathcal{F}$ the Borel $\sigma$-algebra of $B(\mathbb{R}_{\infty})$ we obtain uniqueness of the measure given by \cref{theo: radon and smooth in R}.(2).
    
\end{remark}

\section{Radon measures on the Berkovich analyfication and generically stable measures on $\mathrm{ACVF}$}\label{sec:radon-barycentres-stable-completion}
The purpose of this section is to establish a one to one correspondance between the finite Radon measures on the stable completion $\widehat{X}(M)$ and generically stable Keisler measures localized in $X$; where $X$ is a projective variety defined over $M \models \mathrm{ACVF}$ a spherically complete model whose value group is $\mathbb{R}$. \\

\subsection{From Borel measures in the stable completion to generically stable measures}
We recall some of the terminology of Hrushovski-Loeser in \cite[Sections~3.1, 3.4 and 14.1]{tamegeo}. Let $M \models \mathrm{ACVF}$ and let $X$ be a $M$-definable set. 
\begin{definition}
\label{def:stable-completion}
The \emph{stable completion} of $X$ over $M$, denoted $\widehat X(M),$
is the set of $M$-definable stably dominated/ generically stable types concentrating on
$X$. Equivalently, by \cite[Proposition~2.9.1] {tamegeo}, $\widehat X(M)$ consists of the global types that are $M$-definable, concentrating on $X$ and are orthogonal to the value group. 
\end{definition}

By \cite[Theorem~3.1.1]{tamegeo}, the functor
$M\mapsto\widehat X(M)$
is represented by a canonical strict pro-definable set, also denoted
$\widehat X$. We now describe briefly the topology one can induce over $\widehat{X}(M)$ whenever $X \subseteq M^{n}$.
\begin{definition}
\label{def:hl-topology-stable-completion} Let $f(x)$ be a fixed a regular function on $\mathbb{A}^{n}$ defined on a Zariski open set $W$. By orthogonality to $\Gamma_{\infty}(M)$ given $q \in \widehat{W}(M)$, the pushforward of $q$ along $v\circ f$ is a realized type in $\Gamma_{\infty}(M)$. We write $v(f(q))$ to indicate such element (whenever $q(x)\vdash f(x)=0$ then $v(f(q))=\infty$). Then each regular function $f(x)$ induces a natural map $r_{f}: \widehat{W\cap X}(M) \to \Gamma_{\infty}(M)$ given by $q$ goes to $v(f(q))$. And we define the topology on $\widehat{X}(M)$ to be the coarsest topology making such maps continuous, i.e. given $I$ an $M$-definable open subset of $\Gamma_{\infty}(M)$ the set $\mathbb{O}_{W,f,I}=\{ q \in \widehat{X}(M) \ | q \vdash x \in W \wedge \ v(f(q)) \in I\}$, is an open set of $\widehat{X}(M)$. Indeed, the family of sets $\mathbb{O}_{W,f,I}$ \footnote{where $W$ varies on the Zariski open sets, $f$ varies over the regular functions on $W$, and $I$ over the open $M$-deifnable intervals of $\Gamma_{\infty}(M)$.} forms a sub-basis for a topology on $\widehat{X}(M)$. We refer to this topology on $\widehat{X}(M)$ as the \emph{Hrushovski-Loeser topology}. 
\end{definition}
We formalize how from a Borel probability measure on the stable completion $\widehat{X}(M)$ we can construct a Keisler measure on $\mathbb{B}(M)$ centered in $X$, when $X$ is a definable subset of the affine space $\mathbb{A}_{n}$. We will then show that such measure must be generically stable.\\

We start by explaining how quantifier elimination in $\mathrm{ACVF}$ implies that the $M$-realizations of all $M$-definable subsets of $D$ are Borel in $\widehat{X}(M)$.

\begin{definition}\label{def:basicdef} Let $M \models \mathrm{ACVF}$ and $X$ be an $M$-definable set.  If $D\subseteq X$ is $M$-definable, we write
$[D]_{\widehat X(M)}
:=
\{p\in\widehat X(M) \ | \ D(x)\in p\}.$
\end{definition}

\begin{lemma}
\label{lem:borel-incidence-over-base}
Let $D\subseteq X$ be $M$-definable. Then
$[D]_{\widehat X(M)}$ is a Borel subset of $\widehat X(M)$.
\end{lemma}

\begin{proof}
By quantifier elimination in $\ACVF$, the set $D$ is a finite Boolean combination of sets of the
form $P(x)=0$ and 
$v(P(x))\leq v(Q(x))$,
where $P$ and $Q$ are polynomials with coefficients in $M$. Let $P \in M[\bar{x}]$. And let $r_{P}: \widehat{X}(M) \to \Gamma_{\infty}(M)$ the map defined by sending $q$ to $v(P(q))$. By the definition of the topology on $\widehat{X}(M)$, $r_{P}$ is continuous. In particular  
the sets 
\begin{align*}
[P(x)=0]_{\widehat X(M)}&=v(P)^{-1}(\{\infty\}) \ \text{and}\\
[v(P)\leq v(Q)]_{\widehat X(M)}
&=
\{q \in \widehat{X}(M):v(P(q))\leq v(Q(q))\} \end{align*}
are closed subsets of $\widehat X(M)$. It follows  that
$[D]_{\widehat X(M)}$
is Borel in $\widehat{X}(M)$. \end{proof}
Using \cref{lem:borel-incidence-over-base} we can, given any Borel probability measure $\mu$ on $\widehat X(M)$, define a Keisler measure over $M$ concentrating on a definable set $X$ as follows:

\begin{definition}
\label{def:keislermeasure}
Let $\nu$ be a Borel probability measure on $\widehat X(M)$. We
define the map $\mathbb{B}_{x}(M) \mapsto [0,1]$ as 
$\mu_{\nu,M}(D)
:=
\nu\bigl([D]_{\widehat X(M)}\bigr)$
for every $M$-definable set $D\subseteq X$.
\end{definition}

\begin{lemma}
\label{lem:local-radon-barycentre-is-measure}
The map $\mu_{\nu,M}$ is a  countably additive Keisler measure over $M$.
\end{lemma}

\begin{proof}
By Lemma~\ref{lem:borel-incidence-over-base}, the map is
well-defined. The function
$D\mapsto[D]_{\widehat X(M)}$
respects Boolean operations. Indeed, since the elements of
$\widehat X(M)$ are complete types we have that for any definable sets $D,E\subseteq X$ 
$[D\cap E]_{\widehat X(M)}=[D]_{\widehat X(M)}\cap[E]_{\widehat X(M)}$. Similarly, $[X\setminus D]_{\widehat X(M)}=\widehat X(M)\setminus[D]_{\widehat X(M)}$. Lastly, given $(E_{i})_{i \in \mathbb{N}}$ a countable family of pairwise disjoint $M$-definable sets, such that their union $E=\bigcup_{i \in \mathbb{N}}E_{i}$ is also $M$-definable, then 
$[E]_{\widehat{X}(M)}=[\bigcup_{i \in \mathbb{N}} E_{i}]_{\widehat X(M)}
=\bigcup_{i \in \mathbb{N}}[E_{i}]_{\widehat X(M)}$.  Countable additivity of $\mu_{\nu,M}$ therefore follows from  $\sigma$-additivity of $\nu$, and  $\mu_{\nu,M}(X)=\nu(\widehat X(M))=1$. 
\end{proof}

We will also need to consider the pushforward of the Keisler measure constructed in \cref{lem:local-radon-barycentre-is-measure} via a definable map onto the valued group. Let $\mathcal L_\Gamma:=\{+, \leq, 0 \}$ and  $B\subseteq \Gamma_\infty(\U)$ a set of parameters. We write $\mathbb B_y^{\Gamma}(B)$ to denote the boolean algebra of definable sets in the structure $(\Gamma_\infty(\U), +, \leq , 0)$ with parameters from $B$.  

\begin{definition}\label{def:localizedpushforwards} Let $g \colon X \mapsto \Gamma_{\infty}^{r}$ be an $M$-definable map. Let $\lambda_{\nu,g}$ be the Keisler measure obtained on $\mathbb{B}^{\Gamma}_{r}(M)$ given by taking the pushforward of $\mu_{\nu,M}$ along the map $g$. More precisely,
for every $\Gamma(M)$-definable set $E\subseteq\Gamma_\infty^r$, let  $\lambda_{\nu,g}(E):=\mu_{\nu,M}\bigl(X\cap g^{-1}(E)\bigr)$. We refer to $\lambda_{\nu,g}$ as the \emph{localized pushforward of $\mu_{\nu,M}$ along $g$.}
\end{definition}
Notice that $\lambda_{\nu,g}(\Gamma_\infty^r)=\mu_{\nu, M}(X);$
in particular, $\lambda_{\nu,g}$ is in general a finite measure on $\mathbb B_y^{\Gamma}(\Gamma_{\infty}(M))$
rather than a probability measure. We extend the terminology of
Definition~\ref{def:smooth} to finite measures in the usual way: a nonzero finite measure is called \emph{smooth} if its normalization
is smooth, and the zero measure is smooth by convention.\\

A fundamental, but rather straighforward, fact we will need is the following:
\begin{fact}\label{fact:imageiscountablyadditive} Since the Keisler measure $\mu_{\nu,M}$ on $M$ defined in \cref{def:keislermeasure} is countably additive (\emph{c.f} \cref{lem:local-radon-barycentre-is-measure}), each of the localized pushforwards $\lambda_{\nu,g}$ defined in \cref{def:localizedpushforwards} is also countably additive over $\Gamma_\infty(M)$.
\end{fact}

We now recall the definition of a bounded definable set $X \subseteq M^{n}$ from \cite[Section $4.2$]{tamegeo}.
\begin{definition}\label{def:bounded}Let $M \models \mathrm{ACVF}$ and $X \subseteq M^{n}$ an $M$-definable subset in the valued field sort. We say that $X$ is \emph{bounded} if for some $\gamma \in \Gamma(M)$ then $X$ is contained in $\{(x_{1},\dots, x_{n}): v(x_{i})\geq \gamma, 1 \leq i \leq n\}$.  We can extend this notion for $X \subseteq V$ where $V$ is a projective variety defined over $M$. In such case, there are affine open covers $V=\bigcup_{i \leq n} U_{i}$, and we say that $X$ is bounded if there are bounded subsets $X_{i} \subseteq U_{i}$ such that $X=\bigcup_{i\leq n} X_{i}$. 
\end{definition}

The following two facts illustrate the powerful consequences of quantifier elimination.
\begin{fact}\label{fact:imageofboundedisbounded} If $X$ is an $M$-definable bounded set and $g: M^{n} \rightarrow \Gamma_{\infty}$ is an $M$-definable function, where $g=v\circ f$ for some regular function $f$ on the affine space $\mathbb{A}^{n}$. Then $V=g(X)$ is always bounded in $\Gamma_{\infty}(M)$.
\end{fact}
\begin{remark}\label{rmk:quantifiereliminationisthelaw}By quantifier elimination any definable function from some cartesian power of the valued field sort onto $\Gamma_{\infty}$ is piecewise a difference of functions of the form $\frac{1}{n} f-\frac{1}{m}g$ for $f$ and $g$ regular functions on $\mathbb{A}^{n}$ and $n,m \in \mathbb{N}$ (see for example \cite[Section $3.5$]{tamegeo}).
\end{remark}
\begin{fact}\label{prop:uniquenessbypushforwards}
Let $\nu$ and $\mu$ be two Kiesler measures over $M$. Assume that for every $M$-definable map $g: M^{n}\rightarrow \Gamma_{\infty}(M)$, the pushforward measures along $g$, $\nu_{g}$ and $\mu_{g}$ agree on $\mathbb{B}^{\Gamma}(\Gamma(M))$. Then $\nu=\mu$. Furthermore, by \cref{rmk:quantifiereliminationisthelaw}, we might as well restrict from all definable functions $g$ to only regular functions.

\end{fact}
\begin{proof}
For $(1)$ by quantifier elimination any definable set $D$ is a boolean combination of formulas of the form $P(x)=0$ and $v(Q_{1}(x)) \leq v(Q(x))$ for $P, Q_{1}(x),Q_{2}(x)$ polynomials with coefficients in $M$. Note the first condition is equivalent to $v(P(x))=\infty$. Take $g=v\circ P$. By hypothesis $\mu\big(\{ x : P(x)=0\}\big)=\mu_{g}\big(\{\infty\}\big)=\nu_{g}\big(\{\infty\}\big)=\nu \big( \{ x : P(x)=0\big).$\\
Similarly, let $g: M^{n}\to \Gamma_{\infty}^{2}(M)$ be the map given by sending $x$ to $(v(Q_{1}(x)),v(Q_{2}(x)))$. Write $D=\{ (\gamma,\beta) \in \Gamma_{\infty}(M)^{2} \ | \ \exists x \in M \ v(P(x))=\gamma \wedge v(Q(x))=\beta\}$. Likewise,\\
$ \mu\big(\{ x : v(Q_{1}(x)) \leq v(Q_{2}(x))\}\big)=\mu_{g}(D)=\nu_{g}(D)=\nu \big(\{ x : v(Q_{1}(x)) \leq v(Q_{2}(x))\}.$
\end{proof}

\begin{theorem}
\label{thm:radon-barycentre-criterion} Let $M \models \mathrm{ACVF}$ and assume that $\Gamma(M)=(\mathbb R,+, \leq, 0)$. Let $X \subseteq M^{n}$ an $M$-definable set in the valued field sort that is bounded. Let $\nu$ be a Borel probability measure on $\widehat{X}(M)$, equipped with the Hrushovski Loeser topology. Then, for every finite tuple
$e\in\Gamma(\U)^m$, the measure $\mu_{\nu,M}$ has a unique extension to $Me$, so that every global $M$-invariant extension of
$\mu_{\nu,M}$ is orthogonal to $\Gamma$.  In particular, $\mu_{\nu,M}$ has a
unique global $M$-invariant extension $\mu_\nu$ which is
generically stable over $M$.
\end{theorem}
\begin{proof} We start with the following claim.
\begin{claim} The measure $\mu_{v,M}$ has a unique extension for every finite tuple $e \in \Gamma(\U)^{m}$.
\end{claim}
\begin{claimproof} Fix a finite tuple
$e\in\Gamma(\U)^m.$
Let $\eta_1$ and $\eta_2$ be two extensions of $\mu_{\nu,M}$ to
$Me$. Let $\phi(x,y)$ be a $\mathcal{L}(M)$ formula where $y$ is a variable ranging in $\Gamma$. Thus  $\varphi(x;e)$ is an arbitrary $\mathcal{L}(Me)$-formula. 
 By \cref{lem:finite-gamma-normal-form}, there are an
$M$-definable set
$S\subseteq\Gamma^r,$ an $M$-definable map
$g:X\mapsto S,$
and a formula $\theta(u;y)$
in the pure language of divisible ordered abelian groups, with
parameters from $\Gamma(M)$, such that
\[\U \models  \forall x \in X \forall y \in \Gamma^{m} \ \big(  \varphi(x;y)
\leftrightarrow
\theta(g(x);y)\big).\] 
In particular, $\U \models \forall x \in X \big(\phi(x,e)\leftrightarrow \theta(g(x),e)\big)$. 
We will now take pushforwards of the measures under $g$. Let $\lambda_{\nu,g}$ be the pushforward of $\mu_{\nu, M}$ under $g$, and let $\lambda_1$ and $\lambda_2$ be the pushforwards of $\eta_1$ and $\eta_2$ under
$g$. Then $\lambda_{\nu,g}$ is a Keisler measure on $\mathbb B^{{\Gamma}}_y (\mathbb R_{\infty})$ and $\lambda_1$ and $\lambda_2$ are Keisler measures on $\mathbb B^{\Gamma}_y (\mathbb R_{\infty}\cup \{e\})$ extending $\lambda_{\nu,g}$. Since $X$ is bounded, by \cref{fact:imageofboundedisbounded} so it is $V=g(X)$. By \cref{fact:imageiscountablyadditive} the measure $\lambda_{g,\nu}$ is a countably additive measure on $V$. By \cref{theo: radon and smooth in R}.(1) we have $\lambda_1=\lambda_2$.

Let
$E_e:=\{g(x)\in S\ | \ \theta(g(x);e) \ \text{and}\ x\in X \}$ hence $\eta_1(\varphi(x;e))
=
\lambda_1(E_e)
=
\lambda_2(E_e)
=
\eta_2(\varphi(x;e)).$
Since $\varphi(x;e)$ was arbitrary, we conclude that
$\eta_1=\eta_2.$
Thus $\mu_{\nu,M}$ has a unique extension to $Me$. \end{claimproof} 

\medskip

 By
\cref{fact:global-invariant-extension}, we can choose a global $M$-invariant extension
$\mu_\nu$ of $\mu_{\nu,M}$. By Claim~$1$ for every finite
tuple $e$ from $\Gamma(\U)$, the restriction $\mu\rest_{M}=\mu_{\nu,M}$
has a unique extension to $Me$. Hence, $\mu_\nu$ 
is orthogonal to $\Gamma$.
By \cref{thm:main-generically-stable-measures-acvf}, $\mu_\nu$ is generically stable over $M$, therefore $M$-invariant and finitely satisfiable over $M$. By \cite[Proposition~3.3]{genmeasures}, a generically stable
$M$-invariant measure is the unique global $M$-invariant extension
of its restriction to $M$. So $\mu_\nu$ is the unique
global $M$-invariant extension of $\mu_{\nu,M}$ and it is generically stable, as required.
\end{proof}

\subsection{From generically stable measures to Radon measures on the Berkovich analyfication}
In this section, we will show that for a model $M \models \mathrm{ACVF}$ which is spherically complete and with value group $(\R, +, \leq, 0)$ any generically stable measure $\mu$ on a projective variety $X$ defined over $M$ induces a Radon measure on its Berkovich analyfication $X^{an}$, as defined in \cite{Berkovich}.\\
We recall the Hrushovki-Loeser topology on $\widehat{X}(M)$ over projective varieties in a compatible way with their sheaf structure (this generalizes \cref{def:hl-topology-stable-completion} which only captures the affine case). 
\begin{definition}\label{def:hl-topology-stable-completion proj} Let $X$ be an algebraic variety defined over $M$, and let $\mathcal{O}_{X}$ be the sheaf structure of $X$. Let $W\subseteq V$ be an $M$-definable Zariski open subset
and let $f \in \mathcal{O}_{X}(W)$ (i.e. a regular function in the affine case and a regular function on a chart in the projective case). For
$p\in\widehat W(M)$, orthogonality to $\Gamma$ implies that the
pushforward of $p$ under
$x\mapsto v(f(x))$ is an element of $\Gamma_\infty(M).$ We denote
this point by $v(f(p)).$ The topology on $\widehat X(M)$ is the coarsest topology for which all
maps
$v(f)\colon  
\widehat W(M)\mapsto \Gamma_\infty(M)$
are continuous, for all $W$ Zariski open on $X$ and $f \in \mathcal{O}_{X}(W)$. Equivalently, the sets
$\{
p\in\widehat W(M) \ | \ 
v(f(p))\in I
\}_{I,f,W}$
where $I\subseteq\Gamma_\infty(M)$ is open in the order topology, $W\subseteq X$ is Zariski open and $f \in \mathcal{O}_{X}(W)$ form a sub-basis. We refer to this topology as the \emph{Hrushovski-Loeser topology}. 
\end{definition}
One of the model theoretic view points of the work of Hrushovski-Loeser, is that for $X$ is a projective variety, defined over a spherically complete model $M$ of $\mathrm{ACVF}$ with value group $(\mathbb{R},+,\leq, 0)$, the stable completion $\widehat X(M)$ is homeomorphic to the Berkovich analyfication $X^{an}$ defined in \cite{Berkovich}. More precisely, 
\begin{fact}
\label{fact:stable-completion-berkovich}
Assume that $M$ is a spherically complete algebraically closed valued field and that $\Gamma(M)=(\mathbb R,+,\leq, 0).$
Let $X$ be an algebraic variety over $M$. Then there is a canonical
homeomorphism
$\widehat X(M)
\ \simeq\
X^{\mathrm{an}},$
where $X^{\mathrm{an}}$ denotes the Berkovich analytification of $X$.
\end{fact}
\begin{proof}
This is \cite[Lemma~14.1.1]{tamegeo}.
\end{proof}
It is of course enough to prove the theorem for the projective space $\mathbb P^n$ in which $X$ embeds into. A key component in this direction is Payne's tropicalization theorem \cite[Theorem 4.2]{Payne09}, which states: 
\begin{fact}\label{fact: Payne}
    Let $X$ be a quasi-projective variety over a field $K$. Then the Berkovich analytification $X^{\mathrm{an}}$ of $X$ is 
    homeomorphic to the inverse limit $\varprojlim \mathbf{Trop}(X,\iota )$, where the limit is taken over all quasiprojective toric embeddings of $X$. (\emph{c.f.} \cite[Definition $4.1$]{Payne09}).  The homeomorphism is given 
     via the inverse map $\varprojlim \pi_\iota$ where 
     $\varprojlim \pi_\iota$ is a natural proper and continuous map from $X^{\mathrm{an}}$ to $\mathbf{Trop}(X,\iota)$.
\end{fact}
Each tropical image $\mathbf{Trop}(X,\iota)$ is defined after fixing a toric variety $Y=Y(\Delta)$ where $\Delta$ is a fan and $\iota$ is a closed embedding from $X$ to $Y(\Delta)$. For details on the construction we refer the reader to \cite[Section $3$]{Payne09}.\\
There are two reasons why we will not dwell on the the precise definitions of $\mathbf{Trop}(X,\iota)$ and the maps between them under which Payne's homeomorphism is valid. 
The first one, is that we will only use Fact \ref{fact: Payne}
for $X$ being the projective space $\mathbb P^n$, and we can therefore define $\varprojlim \mathbf{Trop}(\mathbb{P}^n,\iota )$ through its affine coverings.\\
Recall that given an algebraically closed field $M$ we define $\mathbb P^n(M)$
by taking the quotient space $M^{n+1}\setminus\{(0,\ldots,0)\}$, by the  equivalence relation $(a_0,\ldots,a_n)\sim (b_0,\ldots,b_n)$
if there exists \(k\in M^\times\)  such that $b_i=ka_i$ for every \(0\leq i\leq n\).  Then  $\mathbb P^n(M)$ is covered by the open spaces $U_i(M):=\{[(a_0,\ldots,a_n)]_\sim\mid a_i\neq 0\}$ and that via the standard chart $U_i(M)$ is isomorphic to $\mathbb A^n(M)$, more explicitly   $[(a_0,\ldots,a_n)]_\sim\mapsto (\frac{a_0}{a_i},\ldots,\frac{a_{i-1}}{a_i}, \frac{a_{i+1}}{a_i},\dots, \frac{a_n}{a_i})$.\\
Although the space $\mathbf{Trop}(\mathbb P^n,\iota)$ is not hard to define through homogeneous polynomials and it is beautifully done in \cite{SturmfelsMalagan}, for the purposes of this paper, it is better to think of $\mathbf{Trop}(\mathbb P^n,\iota)$ and $\varprojlim \mathbf{Trop}(\mathbb P^n,\iota)$ through the tropicalization of the affine spaces $U_i$. Now we state the second reason for which we will not need to deal with the inverse limit $\varprojlim \mathbf{Trop}(X,\iota )$ in its more general form: as explained in Chapter 4 of \cite{Payne09}, the analytifications $U_i^{an}$ cover $(\mathbb P^n)^{an}$ and if for any $\iota$ we take the restriction of $\pi_\iota$ to $(U_i)^{an}$ then the maps $\pi_\iota|_{U_i^{an}}$ to $\mathbf{Trop}(U_i,\pi_\iota|_{U_i^{an}})$ glue together to define the original map $\pi_{\iota}: \mathbb P^n \to \mathbf{Trop}(\mathbb P^n,\iota)$. Thus we can concentrate in the tropicalizations of the affine space $\mathbb A^{an}$. 
\begin{definition}\label{def:affine-tropicalization}
The \emph{standard affine tropicalization map} is the coordinatewise
valuation map $\mathrm{trop}_{\mathrm{aff}}\colon
\mathbb A^n(M)\longrightarrow\mathbb R_\infty^n$ defined by 
$(x_1,\ldots,x_n)
\to
(v(x_1),\ldots,v(x_n))$. 

More generally, let \(M[X]\) be  the coordinate ring of regular functions on $\mathbb A^n$, where $X=(x_{1},\dots, x_{n})$. Let
$\mathbf f=(f_1,\ldots,f_r)$ be a finite tuple of regular functions
on $\mathbb A^n$. Let $Y_{\mathbf{f}}$ be the Zariski open set in which all the functions $f_{i}$ are defined. The tuple $\mathbf f$ determines a morphism
$\iota_{\mathbf f}\colon Y_{\mathbf{f}}\to\mathbb A^r$, given by
$x\mapsto(f_1(x),\ldots,f_r(x))$, and hence an associated
\emph{tropicalization map} $ \mathrm{trop}_{\mathbf f}
\colon  Y_{\mathbf{f}}(M)\to \mathbb R_\infty^r$, given by taking the composition $\mathrm{trop}_{\mathrm{aff}}\circ\iota_{\mathbf f}$.\\
More explicitly,
$\mathrm{trop}_{\mathbf f}(x)
=(v(f_1(x)),\ldots,v(f_r(x)))$.
We define the
\emph{tropicalization of $\mathbb A^n$ with respect to $\iota_{\mathbf f}$} by
$\operatorname{Trop}(\mathbb A^n,\iota_{\mathbf f})
:=
\overline{
\{\mathrm{trop}_{\mathbf f}(x)\mid x\in Y_{\mathbf{f}}(M)\}
},$
where the closure is taken in $\mathbb R_\infty^r$, equipped with the product topology. 
\end{definition}

These affine tropicalization maps extend naturally to the Berkovich analytification
\((\mathbb{A}^n)^{\mathrm{an}}\). Recall that a point \(p \in \mathbb (\mathbb{A}^n)^{\mathrm{an}}\)
is represented by a multiplicative semi-norm \(|\cdot|_p\) on \(M[X]\) 
extending the absolute value \(a\mapsto\exp(-v(a))\) on \(M\).
Thus, for any tuple of regular functions \(\mathbf f=(f_1,\ldots,f_r)\), we obtain a continuous map
\[
\pi_{\mathbf f}\colon
(\mathbb A^n)^{\mathrm{an}}\longrightarrow
\operatorname{Trop}(\mathbb A^n,\iota_{\mathbf f}),
\qquad
p\longmapsto
(-\log|f_1|_p,\ldots,-\log|f_r|_p),
\]
where \(-\log 0=\infty\). On the simple points \(\mathbb A^n(M)\), this agrees with
\(\mathrm{trop}_{\mathbf f}\). Let \[\mathcal I_{\mathrm{aff}}(\mathbb A^n)
:=\left\{
(f_1,\ldots,f_r)
\ \middle|\
r\geq1,\;
f_i\in M[X],\;
M[X]=M[f_1,\ldots,f_r]
\right\}.\]
One simplifying aspect about working with $\mathbb A^n$ instead of the projective space $\mathbb{P}^{n}$, is that we can define an order on $\mathcal I_{\mathrm{aff}}(\mathbb A^n)$ given by
\(\mathbf f\preceq\mathbf g\) if \(\mathbf f\) can be obtained from
\(\mathbf g\) by deleting entries.
Then $(\mathcal I_{\mathrm{aff}}(\mathbb A^n), \preceq)$ is directed: for any two generating tuples
\(\mathbf f\) and \(\mathbf g\), their concatenation
\((\mathbf f,\mathbf g)\) is again a generating tuple and is an
upper bound for both. The corresponding tropicalizations therefore
form an inverse system, with transition maps given by coordinate
projections.\\
Note that for any tuple $\mathbf{f}$ of regular functions generating $M[X]$, we have that $\mathrm{Trop}(\mathbb{A}^{n}, \iota_{\mathbf{f}})$ is a definable subset of $\mathbb{R}_{\infty}^{r}$, which inherits the natural topology as a subspace. In particular, by the description of the topology on $\mathbf{Trop}(\mathbb{P}^{n},\iota)$ in \cite[Section $3$]{Payne09} we might view  the set $\mathbf{Trop}(\mathbb{P}^{n},\iota)$ as a definable space in the sense of van den Dries \cite[Chapter $10$]{vdD}. 
\begin{proposition}\label{prop:babytropicalization}
For any tropicalization $\iota$ the set $\mathbf{Trop}(\mathbb P^n,\iota)$ is definably homeomorphic to a compact  subset of $\Ri^k$ for some $k$. In particular, it is contained in a bounded (closed) box. \end{proposition}

\begin{proof}
    By \cite[Theorem 3.4.8]{Berkovich} $(\mathbb P^n)^{\mathrm{an}}$ is compact. And we have a continuous map $\pi_{\iota}: (\mathbb P^n)^{\mathrm{an}} \to \mathbf{Trop}(\mathbb{P}^{n},\iota)$, 
where the image of the simple points in $\mathbb P^n$ is proper and dense in $\mathbf{Trop}(\mathbb{P}^n,\iota)$.  Being a continuous image of a compact space $\mathbf{Trop}(\mathbb{P}^n,\iota)$ is compact. In particular, it is regular as it is compact and Hausdorff space.\\
Recall that the tropical image $\mathrm{Trop}(\mathbb{P}^{n},\iota)$ is defined after fixing a toric variety $Y$ and $\iota: \mathbb{P}^{n} \to Y$ is a closed embedding. By \cite[Example $3.2$]{Payne09}, $\mathrm{Trop}(\mathbb{P}^{n},\iota)$ is the gluing of the $\mathrm{Trop}(U_i,\iota\rest_{U_i})$ for the standard affine covers $U_i$ of $\mathbb P^n$. Let $T_i=\mathrm{Trop}(U_{i},\iota\rest_{U_i})\subseteq\mathbb R_\infty^{r_i}.$
The topology on $\mathbf{Trop}(\mathbb{P}^{n},\iota)$, is defined by gluing the tropicalizations $T_{i}$ (see the discussion before \cite[Example $3.2$]{Payne09}). In particular, the $(T_i)_{i \leq s}$ form a finite open cover of $\mathbf{Trop}(\mathbb{P}^{n}, \iota)$.\\
On the other hand, the spaces $T_{i}$ are tropicalization of the affine space, so each of them is generated by regular functions as in \cref{def:affine-tropicalization}. In particular, each of them is definable. Hence each $T_{i}$ is a definable subset of $\mathbb R_\infty^{r_{i}}$,  which inherits a natural subspace topology. Furthermore, their overlaps are definable, and the transition maps come from monomial localizations, so they are also definable homeomorphisms.\\
Hence, $\mathbf{Trop}(\mathbb P^n, \iota)$ is a definable space in the sense of van den Dries \cite[Chapter $10$]{vdD}, with a definable regular topology. By \cite[Theorem 1.8 in Chapter 10]{vdD} we may assume $\mathbf{Trop}(\mathbb{P}^{n},\iota)$ to be a definable subset of $\Ri^k$ with the order topology for some $k$. Since it is compact, it must be contained in a bounded (closed) box.
\end{proof}

We can now prove the main result of this subsection.

\begin{proposition}\label{prop:projectivecase} Let $M$ be a spherically complete model of $\mathrm{ACVF}$ whose value group is $(\mathbb{R},+, \leq, 0)$. Let $\mu$ be a global $M$-invariant generically stable measure centered in $\mathbb{P}^{n}$. Then there is a unique Radon measure $\hat{\nu}$ on $(\mathbb{P}^{n})^{an}\cong \widehat{\mathbb{P}^{n}}$ which induces $\mu$ (in the sense of \cref{def:keislermeasure}).
\end{proposition}
\begin{proof} 
Let $\mu$ be a generically stable measure concentrating on $\mathbb P^n$. Let $\mathbf{Trop}(\mathbb P^n, \iota)$ be any of the tropicalizations of $\mathbb P^n$ in Payne's inverse limit. By \cref{prop:babytropicalization}, $\mathbf{Trop}(\mathbb P^n, \iota)$ is definably homeomorphic to a compact subset $V_{\iota}$ of $\Ri^k$ for some $k$, which is therefore contained in a box. The pushforward $\mu_\iota$ of $\mu$ on the $\mathbb{R}_{\infty}$-definable set $V_{\iota}$ is generically stable over $\mathbb{R}_{\infty}$, as it is an $M$-definable pushforward of a generically stable measure. By \cref{fact:acvf-stably-embedded-sorts} the value group is distal and stably embedded, so by \cite[Proposition $2.27$]{distal} $\mu_\iota$ is smooth over $\mathbb{R}_{\infty}$. By \cref{theo: radon and smooth in R}.(2) there is a $\sigma$-additive Borel measure $\Lambda_{\iota}$ on $V(\mathbb{R}_{\infty})$ such that $\Lambda_{\iota}$ induces $\mu_{\iota}$ as a Keisler measure on the $\mathbb{R}_{\infty}$-definable subsets of $V$. Note that the system of measures $\Lambda_\iota$ on $\mathbf{Trop}(\mathbb{P}^{n},\iota)$ is by definition compatible with the inverse limit $\varprojlim \mathbf{Trop}(\mathbb P^n, \iota)$. Indeed, all the measures $\mu_{\iota}$ are obtained as the pushforwards of the same Keisler measure $\mu$ on $M$, and they are compatible with the transition maps, and the induced Borel measure $\Lambda_{\iota}$ are unique by \cref{rmk:uniquenessRadon}.\\
All the spaces $\mathbf{Trop}(\mathbb{P}^{n},\iota)$ are compact and Hausdorff, so Prokhorov's Theorem \cite[Theorem $5.5.1$]{burba9} implies that there is a unique Radon measure $\nu$ on $(\mathbb P^n)^{an}\cong \varprojlim \mathbf{Trop}(\mathbb P^n, \iota)$ such that $\pi_{\iota}(\nu)=\Lambda_{\iota}$. \\

It remains just to argue that $\nu$ induces $\mu$ as a Keisler measure. By \cref{fact:stable-completion-berkovich} we can identify $(\mathbb{P}^{n})^{\mathrm{an}}$ with $\widehat{\mathbb{P}^{n}}$ with the Hrushovski-Loeser topology. Let $\mu_{\nu,M}$ be the Keisler measure induced by $\nu$ as described in \cref{def:basicdef}. By construction, $\mu_{\nu,M}$ will have $\mu_{\iota}$ as pushforward on $\mathbf{Trop}(\mathbb{P}^{n},\iota)$, as $\pi_{\iota}(\nu)=\Lambda_{\iota}$ and $\Lambda_{\iota}$ induces $\mu_{\iota}$. This implies that in any of the affine subspaces covering $\mathbf{Trop}(\mathbb{P}^{n},\iota)$, the pushforwards under regular maps are precisely the pushforwards of $\mu$. Hence $\mu_{\nu,M}$ must coincide with $\mu$ by \cref{prop:uniquenessbypushforwards}.

\end{proof}

\begin{corollary}\label{cor:generically-stable-produces-radon}  Let $M$ be a spherically complete model of $\mathrm{ACVF}$ whose value group is $(\mathbb{R}, +, \leq, 0)$. Let $X$ be a projective variety defined over $M$ and let $\mu$ be a global $M$-invariant generically stable measure centered on $X$. Then there is a unique Radon measure $\hat{\nu}$ on the Berkovich analyfication $X^{\mathrm{an}}$ which induces $\mu$ in the sense of \cref{def:keislermeasure}.
\end{corollary}
\begin{proof}
Choose an $M$-defined projective embedding
$j\colon X\hookrightarrow\mathbb P^n$.
By \cite[Fact~5.1.2.1]{Temkin}, analytification preserves
closed immersions, so we identify $X^{\mathrm{an}}$ with
a closed subspace of $(\mathbb P^n)^{\mathrm{an}}$, and assume that $X$ is a definable (in fact algebraic) subset of $\mathbb P^n$, so that $\mu$ is a global $M$-invariant generically stable measure centered on $\mathbb P^n\supseteq X$.

By \cref{prop:projectivecase}, there is a unique Radon measure $\nu$ on $(\mathbb P^n)^{an}$ such that $\nu_{\nu,M}=\mu$ (c.f. \cref{def:keislermeasure}). 
\end{proof}

\subsection{The correspondance}
We first give a corollary of \cref{thm:radon-barycentre-criterion} for a projective variety. 
\begin{remark}\label{rmk:boundedinprojective} Let $M$ be a  model of $\mathrm{ACVF}$.\\
For each
$i\leq n$, let $B_i=\left\{[x_0:\cdots:x_n]\in\mathbb{P}^n(M) \ | \ 
x_i\neq 0 \text{ and } v(x_i)\leq v(x_j)\text{ for all }j\right\}.$
This is a subset of the standard affine chart
$U_i$, and recall that $U_{i}$ is definably homeomorphic to the affine space $\mathbb{A}^n$. Via this identification, the set $B_i$ corresponds to $\mathcal{O}^n$, and is therefore
bounded. Moreover,
$\mathbb{P}^n(M)=\bigcup_{i=0}^n B_i,$ because any point has a coordinate of minimal valuation. Consequently, given any definable projective variety $X$ over $M$, 
every $M$-definable subset of a projective variety $X$ is bounded. Indeed one can take the cover of the $B_{i}$ and intersect it with $X$ to obtain a bounded cover of $X$ (in the sense of \cref{def:bounded}), and therefore for every definable subset $D\subseteq X$.
\end{remark}

\begin{corollary}
\label{cor:projective} Let $M \models \mathrm{ACVF}$ and assume that $\Gamma(M)=(\mathbb R,+, \leq, 0)$. Let $X$ be a projective variety defined over $M$.  Let $\nu$ be a Borel probability measure on $\widehat{X}(M)$, equipped with the Hrushovski Loeser topology (\emph{c.f.} \cref{def:hl-topology-stable-completion proj}). Then $\nu$ induces a Keisler measure $\mu_{\nu,M}$ on the boolean algebra of the definable subsets of $X$. Such measure $\mu_{\nu,M}$ has a
unique global $M$-invariant extension $\mu_\nu$ which is
generically stable over $M$.
\end{corollary}
\begin{proof} Choose an $M$-definable projective embedding $j \colon X \to \mathbb{P}^{n}$ and let $\mathbb{P}^{n}=\bigcup_{i \leq n} B_{i}$ be the bounded cover of $\mathbb{P}^{n}$ as in \cref{rmk:boundedinprojective}. Recall that $B_{i} \subseteq U_{i}$, where $U_{i}$ is the standard affine cover. Then $X=\bigcup_{i=1}^{n} X_{i}$ where each $X_{i}\colon= X \cap B_{i} \subseteq X \cap U_{i}=W_{i}$. In particular,  $X_{i}$ is a bounded subset of $U_{i}$. For each $i \leq n$, let $\phi_{i}$ the definable homeomorphism of $U_{i}$ to the affine space $\mathbb{A}^{n}$, via this identification we identify $X_{i}$ with a bounded subset $O_{i}$ of $\mathbb{A}^{n}$. This identification extends to their stable completions $\widehat{X}_{i}$ to $\widehat{O}_{i} \subseteq \widehat{\mathbb{A}}^{n}$. \\
For each $i \leq n$ let $\nu_{i}$ be the restriction of $\nu$ to $\widehat{X}_{i}$, since $\widehat{B}_{i}$ is Borel within $\widehat{X}$ the restriction $\nu_{i}$ is still Borel on $\widehat{X}_{i}$ (equipped with the subspace topology). Let $\eta_{i}$ be the Borel measure on $\widehat{O}_{i}$, obtained by taking the pushforward of $\nu_{i}$ along $\phi_{i}$. Finally, let $\mu_{\eta_{i},M}$ be the Keisler measure induced by $\eta_{i}$ on the Boolean algebra of $M$-definable subsets of $O_{i}$ (as in \cref{def:basicdef}). By \cref{thm:radon-barycentre-criterion}, $\mu_{\eta_{i},M}$ has a unique global $M$-invariant extension $\mu$ which is generically stable over $M$.\\
Note that this naturally induces a Keisler measure $\mu_{i,M}$ on $X_{i}$: given a definable subset $D_{i} \subseteq X_{i}$ the set $\phi_{i}(D_{i})$ is a definable subset of $O_{i}$ and hence $\mu_{i}(D_{i})=\mu_{\eta_{i},M}(\phi_{i}(D_{i}))$. So $\mu_{i}$ is also $M$-definable and finitely satisfiable over $M$.\\

We will now define a Keisler measure on $X$: given a definable subset $D$ of $X$, we have that $D=\bigcup_{i=1}^{n} D_{i}$ where $D_{i}=D \cap X_{i}$. For $i \neq j$, let $D_{ij}=D_{i} \cap D_{j}$. The following claim implies that the measures on $D_i$ induce a unique measure on $X$ using inclusion-exclusion, completing the proof.
\begin{claim}
    $\mu_{i}(D_{ij})=\mu_{j}(D_{ij}).$
    \end{claim} 
\begin{claimproof}
By definition, $\mu_{i}(D_{ij})=\mu_{\eta_{i},M}(\phi_{i}(D_{ij}))$. And the later is the Keisler measure induce by $\eta_{i}$, obtained by taking the pushforward of $\nu_{i}$ along $\phi_{i}$. Similarly for the index $j$. Since the charts $\phi_{j}$ are compatible, then given a Borel subset $S \subseteq X_{i} \cap X_{j}$ we have that $\phi_{i}(S)=\phi_{j}(S)$. In particular, $\eta_{i}$ and $\eta_{j}$ agree on the measure of such sets, i.e. $\eta_{i}(\phi_{i}(S))=\eta_{j}(\phi_{j}(S))$. In particular, the set $D_{ij} \subseteq X_{i}\cap X_{j}$ is definable, so they are $\phi_{i}(D_{ij})=\phi_{j}(D_{ij})$. Consequently, $\mu_{\eta_{i},M}(D_{ij})=\mu_{\eta_{j},M}(D_{ij})$, as required.\end{claimproof} \\
The main statement then follows.  
\end{proof}

\begin{theorem}
\label{thm:radon-generically-stable-correspondence}
Let $M\models\ACVF$ be spherically complete with
$\Gamma(M)=(\mathbb R,+,\leq ,0),$
and let $X$ be a projective $M$-variety. Under the canonical
homeomorphism
$\widehat X(M)\simeq X^{\an},$
there is a canonical  bijection between Radon probability measures on the Berkovich analytification
of $X$ and the global $M$-invariant generically stable measures concentrating
on $X$.
\end{theorem}
\begin{proof}
This follows from \cref{cor:projective} and \cref{cor:generically-stable-produces-radon}. 
\end{proof}

\noindent\textsc{Alf Onshuus. Departamento de Matemáticas, Universidad de los Andes, Carrera 1 no. 18A - 12, Edificio H, Bogotá, Colombia.}

Email address:\textsf{aonshuus@uniandes.edu.co}.\medskip

\noindent\textsc{Mariana Vicar\'ia. University of Chicago, Department of Mathematics, Eckhart Hall, 5734 S University Ave, Chicago, IL, 60637.}

Email address: \textsf{mvicaria@uchicago.edu}.
\end{document}